\documentclass[reqno]{amsart}

\usepackage{amsmath,amssymb,amsthm,amsfonts,enumerate, enumitem,hyperref,cleveref}
\usepackage{dsfont}
\usepackage[all]{xy}
\usepackage[T1]{fontenc}
\usepackage{tikz}
\usepackage{pgfplots}
\pgfplotsset{compat=1.15}
\usepackage{mathrsfs}
\usepackage{tikz-cd}
\usetikzlibrary{arrows}

\usepackage{mathtools}

\usepackage{amsmath,environ}
\NewEnviron{splii}{%
\begin{equation*}\begin{split}
  \BODY    
\end{split}\end{equation*}
}

\DeclareMathOperator{\dom}{dom}

\newcommand{\Z}{\mathbb{Z}}

\newcommand{\K}{\mathbb{K}}
\newcommand{\C}{\mathscr{C}}
\newcommand{\D}{\mathscr{D}}

\newcommand{\datum}[1]{M\mathbf{-Datum}_{#1}}
\newcommand{\act}[1]{M\mathbf{-Act}_{#1}}
\newcommand{\pact}[1]{M\mathbf{-pAct}_{#1}}
\newcommand{\spact}[1]{M\mathbf{-spAct}_{#1}}
\newcommand{\Hom}[1]{\mathrm{Hom}_{#1}}

\DeclareMathOperator{\Aut}{Aut}

\DeclareMathOperator{\Span}{\mathbf{Span}}
\DeclareMathOperator{\spa}{\mathbf{span}}
\DeclareMathOperator{\Par}{\mathbf{Par}}
\DeclareMathOperator{\pa}{\mathbf{par}}

\DeclareMathOperator{\Set}{\mathbf{Set}}
\DeclareMathOperator{\Top}{\mathbf{Top}}

\newtheorem{teorema}{Theorem}[section]
\newtheorem{corolario}[teorema]{Corollary}
\newtheorem{lema}[teorema]{Lemma}
\newtheorem{proposicao}[teorema]{Proposition}

\theoremstyle{definition}
\newtheorem{definicao}[teorema]{Definition}
\newtheorem{exemplo}[teorema]{Example}

\newtheorem{rem}[teorema]{Remark}

\crefrangeformat{equation}{#3(#1)#4--#5(#2)#6}

\crefname{teorema}{Theorem}{Theorems}
\crefname{lema}{Lemma}{Lemmas}
\crefname{corolario}{Corollary}{Corollaries}
\crefname{proposicao}{Proposition}{Propositions}
\crefname{definicao}{Definition}{Definitions}
\crefname{exemplo}{Example}{Examples}
\crefname{rem}{Remark}{Remarks}
\crefname{section}{Section}{Sections}

\crefname{equation}{\unskip}{\unskip}
\crefname{enumi}{\unskip}{\unskip}

\newcommand{\af}{\alpha}
\newcommand{\bt}{\beta}

\newcommand{\vf}{\varphi}

\renewcommand{\iff}{\Leftrightarrow}

\newcommand{\red}[1]{#1}
\newcommand{\blue}[1]{#1}

\begin{document}
        \title[Partial monoid actions]{Partial monoid actions on objects in categories with pullbacks and their globalizations}	
	
	\author{Mykola Khrypchenko}
	\address{Departamento de Matem\'atica, Universidade Federal de Santa Catarina,  Campus Reitor Jo\~ao David Ferreira Lima, Florian\'opolis, SC, CEP: 88040--900, Brazil \and CMUP, Departamento de Matemática, Faculdade de Ciências, Universidade do Porto,
		Rua do Campo Alegre s/n, 4169--007 Porto, Portugal}
	\email{nskhripchenko@gmail.com}
	
	\author{Francisco Klock}
	\address{Departamento de Matem\'atica, Universidade Federal de Santa Catarina, Florian\'opolis, SC, CEP: 88040-900, Brazil}
	\email{francisco\_gabriel25@hotmail.com}
 
	
	\begin{abstract}
		Let $M$ be a monoid, $\C$ a category with pullbacks and $X$ an object of $\C$. We introduce the notion of a partial action $\af$ of $M$ on $X$ and study the globalization question for $\af$. If $\af$ admits a reflection in the subcategory of global actions, then we reduce the problem to the verification that a certain diagram is a pullback in $\C$. We then give a construction of such a reflection in terms of a colimit of a certain functor with values in $\C$. We specify this construction to the case of categories admitting certain coproducts and coequalizers.
	\end{abstract}
	
	\maketitle
	
	\tableofcontents
	
	\section*{Introduction}
	
Partial group actions are a generalization of classical (global) group actions that deals with partial symmetries of an object. The systematic study of this notion began in the middle of the 1990s in \cite{exel1994acao-parcial-origem,McClanahan95,exel_partial_1998}, although the concept itself had appeared earlier in various areas of mathematics under different names and in different generalities, for example, in differential geometry~\cite{Palais} and topology~\cite{Megrelishvili86}. Partial group actions have also been known as an ingredient in the description of $E$-unitary inverse semigroups~\cite{McAlister74-II,Petrich-Reilly79}. Notice that Green and Marcos used the name ``partial action'' for a related notion in~\cite{Green-Marcos94}. We refer the reader to the survey papers~\cite{Dokuchaev-survey,Batista-survey} containing an extensive literature on the subject.

A lot of examples of partial actions come from restrictions of global actions, so it is natural to ask whether any partial action can be obtained this way. This is always the case for partial group actions on sets~\cite{Abadie03,kellendonk2004partial}. The corresponding global action is called a \textit{globalization} or an \textit{enveloping action} of the given partial action. The question of existence of a globalization was studied in the context of groups acting partially on topological spaces by Abadie~\cite{Abadie03} and by Kellendonk and Lawson~\cite{kellendonk2004partial}, on cell complexes by Steinberg~\cite{Steinberg2003}, on (associative) algebras by Dokuchaev and Exel~\cite{DE}, on $s$-unital rings by Dokuchaev, del R\'io and Sim\'on~\cite{DRS}, on semiprime rings by Ferrero and coauthors~\cite{Ferrero06,Cortes-Ferrero09,Bemm-Ferrero13}, on abstract algebraic systems (with applications to the case of semigroups) by Khrypchenko and Novikov~\cite{KN16}, on ($R$-linear) semicategories~\cite{Cortes-Ferrero-Marcos16} by Cortes, Ferrero and Marcos.

Megrelishvili and Schr\"oder~\cite{Megrelishvili_2004} were the first who introduced and studied partial monoid actions. They showed that strong and \textit{confluent} (a notion depending on the monoid presentation) partial actions of a monoid on a topological space are globalizable. Hollings~\cite{PartMonoids} considered arbitrary partial monoid actions on sets and proved that the strong ones are exactly the globalizable ones (his globalization also respects a fixed set of generators of $M$ in the spirit of~\cite{kellendonk2004partial}). Kudryavtseva and Laan~\cite{Kudryavtseva-Laan23} studied two different notions of a globalization of a partial semigroup action on a set, one of them being the initial object in the corresponding category of globalizations (its construction generalizes that of~\cite{Megrelishvili_2004,PartMonoids}), and the other one being the final object in the same category. We point out that partial monoid actions play an important role in the description of proper two-sided restriction semigroups~\cite{Cornock-Gould,Kudryavtseva-PMA}.

Hu and Vercruysse~\cite{geopactions} proposed a unified approach to treat various types of partial actions (and coactions) by introducing partial actions of algebras in monoidal categories with pullbacks. This naturally leads to the globalization problem that was studied by Saracco and Vercruysse in~\cite{saracco2022globalization}, where the globalizable partial actions were characterized in terms of equalizers and pushout diagrams in the corresponding categories. In~\cite{saracco2022globalization2}, the general globalization theorem was specified to the case of partial actions of topological monoids on topological spaces generalizing the results of~\cite{Abadie03,Megrelishvili_2004}, as well as to the case of partial comodule algebras over bialgebras establishing a connection with the enveloping coaction in the sense of~\cite{Alves-Batista11}. On the other hand, in~\cite{Saracco-Vercruysse23flat}, geometric partial comodules over flat coalgebras in Abelian monoidal categories were shown to be globalizable, recovering as particular cases the results of~\cite{Alves-Batista-Vercruysse19,D'Adderio-Hautekiet-Saracco-Vercruysse22}.

However, there are several classes of partial actions that have been studied in the literature, but not yet covered by the general theory of~\cite{geopactions}. These are, for example, partial group actions on associative rings and algebras, C$^*$-algebras, semigroups, etc. Although the corresponding categories admit a natural monoidal structure, it is not clear, a priori, how to realize a group as a monoid there.

In this paper we develop a theory of partial monoid actions on objects in categories with pullbacks. Observe that we do not assume any monoidal structure on the category under consideration and any relation between the category and the monoid. 

In~\cref{sec-prelim} we fix some notations and recall the notions that will be used in the sequel. \cref{correspondencia-dadosdeacao-set}, which is inspired by \cite[Lemma 1.7]{geopactions}, serves as a motivation to introduce more general concepts below.

Given a monoid $M$ and an object $X$ of a category with pullbacks $\C$, in~\cref{sec-pact} we define the concept of a \textit{partial action datum} and those of a \textit{partial} and a \textit{global action} of $M$ on $X$ together with the suitable notion of a morphism between partial action data to form the categories $\datum\C$, $\pact\C$ and $\act\C$, respectively. 

Any global action of $M$ on $Y$ \textit{restricts} to a partial action of $M$ on $X$, provided that there is a monomorphism $\iota:X\to Y$, as shown in \cref{sec-restr}. In fact, the resulting partial action datum satisfies even a stronger condition than an arbitrary partial action does, motivating us to introduce the concept of a \textit{strong} partial action which is a generalization of that coming from~\cite{Megrelishvili_2004,PartMonoids}. We give \cref{exm-non-strong-pa,exemplo-esp-top-acao-parcial-forte}
showing that not any partial action is strong and not any strong partial action can be obtained as a restriction of a global action. 

\cref{sec-glob} is the main part of the paper. We introduce the notion of a \textit{globalization} $(\bt,\iota)$ of a partial action datum $\af$. Notice that, unlike~\cite{saracco2022globalization}, we do not require the corresponding datum morphism $\iota:\af\to\bt$ to be a \textit{reflection} of $\af$ in $\act\C$. But, whenever a reflection $\iota:\af\to\bt$ of $\af$ in $\act\C$ exists, \cref{equivalencia-glob-universal-glob} gives a criterion for $(\bt,\iota)$ to be a (universal) globalization of $\af$ in terms of a pullback diagram in $\C$ resembling (the dual of) that from~\cite[Theorem 3.5 (II)]{saracco2022globalization}. Although, in general, not any universal globalization is of this form (see \cref{exemplo-glob-sem-reflexao} for a globalizable partial action that admits no reflection in $\act\C$), in many classical cases it is indeed possible to construct a reflection of $\af$ in $\act\C$ out of a colimit of certain functor with values in $\C$ (see \cref{colimite-implica-reflexao}). In particular, if $\C$ admits certain coproducts, then the existence of a reflection of $\af$ in $\act\C$ is equivalent to the existence of a coequalizer of some pair of parallel morphisms in $\act\C$, as proved in \cref{reflector<->coequalizer}. So, we get a condition similar to (the dual of)~\cite[Theorem 3.5 (I)]{saracco2022globalization}.

We finish the paper applying in \cref{sec-pact-Set} our general results to the case $\C=\Set$. As a consequence, we recover in \cref{af-glob-iff-af-strong-for-C=Set} Hollings's result on the globalization of a strong partial monoid action. The volume of the paper does not permit us to include some other classical cases, so they will be published in a separate paper that is currently under preparation.

	\section{Preliminaries}\label{sec-prelim}


\subsection{Partial actions of monoids on sets}

For the entirety of this article, if otherwise is not stated, $M$ will stand for a monoid with identity element $e$.


\begin{definicao}\label{dado-acao-parcial-set-def}
Following \cite{geopactions}, we define a \textit{partial action datum} of $M$ on a set $X$ to be a family of maps $\{\alpha_m : \dom\alpha_m \to X\}_{m \in M}$ where $\dom\alpha_m \subseteq X$, for all $m \in M$.
\end{definicao}

\begin{definicao}\label{acao_parcial_em_conjuntos}
A \textit{partial action} of $M$ on a set $X$ is a partial action datum $\alpha = \{\alpha_m \}_{m \in M}$ of $M$ on $X$, such that:
\begin{enumerate}[label=(PA\arabic*)]
    \item\label{AP1 set} $X_e = X$ and $\alpha_e = id_X$;
    
    \item\label{AP2 set} $\alpha_m^{-1}(\dom\alpha_n) \subseteq \dom\alpha_{nm}$, for all $m,n \in M$;
    
    \item\label{AP3 set} $\alpha_n \circ \alpha_m = \alpha_{nm}$ on $\alpha_m^{-1}(\dom\alpha_n)$, for all $m, n \in M$.
\end{enumerate}
If $\dom\alpha_m=X$ for all $m\in M$, then we say that $\alpha$ is a \textit{global}\footnote{Note that a global action of $M$ on $X$ is the same as a usual (classical) action of $M$ on $X$.} action of $M$ on $X$.
\end{definicao}

\begin{rem}
    Taking $M$ to be a group and denoting $\dom\alpha_m$ by $X_{m^{-1}}$, we see that \cref{AP2 set,AP3 set} are equivalent to (LPA2) and (LPA3) of \cite[Definition~1.2]{geopactions}, respectively. \red{Thus, partial monoid actions of a group $G$ on a set $X$ in the sense of \cref{acao_parcial_em_conjuntos} are exactly \textit{lax partial actions} of $G$ on $X$ in the sense of \cite[Definition~1.2]{geopactions}.} On the other hand, for an arbitrary monoid $M$, \cref{AP2 set,AP3 set} together are equivalent to \cite[Definition~2.2~(PA2')]{PartMonoids}. 

\end{rem}



   
\begin{exemplo}\label{exemplo global}
Given a global action $\beta$ of $M$ on a set $Y$ and $X \subseteq Y$, the following partial action datum
   $$\alpha = \{\alpha_m : \dom\alpha_m = X \cap \beta_m^{-1}(X) \to X\}_{m \in M}, \quad \alpha_m(x) = \beta_m(x), \forall x\in\dom\alpha_m,$$
   is a partial action of $M$ on $X$.
   
\end{exemplo}

\begin{definicao}\label{restriction-definicao-set}
    Let $\beta$ be a global action of $M$ on a set $Y$ and $X \subseteq Y$. The partial action $\alpha$ constructed in \cref{exemplo global} is called the \textit{restriction} of $\beta$ to $X$.
\end{definicao}

\begin{definicao}\label{acao_parcial_forte_em_conjuntos}
A partial action $\alpha = \{\alpha_m \}_{m \in M}$ of $M$ on a set $X$ is said to be \textit{strong} if instead of \cref{AP2 set} we have the following stronger condition:
\begin{enumerate}[label=(PA2')]
    \item\label{AP2 set strong} $\alpha_m^{-1}(\dom\alpha_n) = \dom\alpha_{nm} \cap \dom\alpha_m$, for all $m,n \in M$.
\end{enumerate}
\end{definicao}

\begin{rem}
    The concept of a strong partial monoid action from \cref{acao_parcial_forte_em_conjuntos} is equivalent to that of \cite[Definition~2.4]{PartMonoids}, \cite[Definition~2.3]{Megrelishvili_2004} \red{and \cite[Section 2.1]{saracco2022globalization2} (although in the last two references it is called just a partial action). Observe that the partial monoid actions from \cite[Section 2.1]{saracco2022globalization2} are a particular case of the so-called \textit{(geometric) partial modules} (see \cite[Remark 2.1]{saracco2022globalization2}). As it was proved in \cite[Proposition 2.4]{Megrelishvili_2004}, partial actions of a group $G$ on a set $X$ in the sense of \cite[Definition 1.1]{DE} are exactly strong partial monoid actions of $G$ on $X$ in the sense of \cref{acao_parcial_forte_em_conjuntos}.}
    
    It should be noted that the partial action $\alpha$ from~\cref{exemplo global} is strong. Moreover, it was proved in \cite{PartMonoids} and (in the context of partial monoid actions on topological spaces) in \cite{Megrelishvili_2004} that any strong partial action can be obtained as a restriction of a global action.
\end{rem}

The next notion is a natural generalization of a well-known notion of a morphism of partial group actions (see, for example,~\cite{Abadie03}).

\begin{definicao}
                Given two partial action data $\alpha = \{\alpha_m\}_{m \in M}$ of $M$ on $X$ and $\beta = \{\beta_m\}_{m \in M}$ of $M$ on $Y$, a \textit{datum morphism} from $\alpha$ to $\beta$ is a map $f \colon X \rightarrow Y$ such that
\begin{enumerate}
    \item $f(\dom\alpha_m) \subseteq \dom\beta_m$ for all $m \in M$;
    
    \item $f \circ \alpha_m = \beta_m \circ f$ on $\dom\alpha_m$ for all $m \in M$.
\end{enumerate}
\end{definicao}

\subsection{Spans and partial morphisms}

Given a category $\C$, we will use the same \red{character} $\C$ to denote the class of objects of $\C$. For any pair of objects $X, Y \in \C$ we will denote by $\Hom{\C}(X,Y)$ the collection of morphisms from $X$ to $Y$ in $\C$. As usual, $id_X$ will mean the identity morphism at $X \in \C$ and $g\circ f:X\to Z$ the composition of $f:X\to Y$ and $g:Y\to Z$. 

For the remainder of this section, fix a category $\C$ and $X,Y\in\C$. 

\begin{definicao}
    A \textit{span}~\cite{Benabou-bicategories} from $X$ to $Y$ is a triple $(A,f,g)$ where $A \in \C$ and $f: A \rightarrow X$ and $g: A \rightarrow Y$ are morphisms, as illustrated.
    
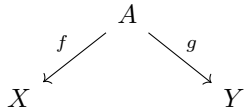
\begin{center}
    \begin{tikzcd}
  & A \arrow[ld, "f"'] \arrow[rd, "g"] &   \\
X &                                    & Y
\end{tikzcd}
\end{center}
\end{definicao}

\begin{definicao}
        Given $(A,f,g)$ and $(B,h,k)$ spans from $X$ to $Y$, a \textit{morphism of spans} from $(A,f,g)$ to $(B,h,k)$ is a morphism $\varphi \colon A \rightarrow B$ in $\mathscr{C}$ such that the following diagram commutes.
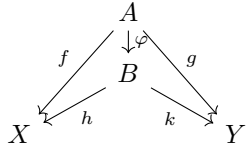
\begin{center}
   \begin{tikzcd}[ampersand replacement=\&, row sep = small]
  \& A \arrow[ldd, "f"'] \arrow[rdd, "g"] \arrow[d, "\varphi"] \&   \\
  \& B \arrow[ld, "h"] \arrow[rd, "k"']                       \&   \\
X \&                                                          \& Y
\end{tikzcd}
\end{center}
\end{definicao}

The spans from $X$ to $Y$ form a category $\Span_{\C}(X,Y)$, whose morphisms are morphisms of spans and the composition is inherited from $\C$. Denote by $[A,f,g]$ the isomorphism class of a span $(A,f,g)$ from $X$ to $Y$, and let $\spa_{\C}(X,Y) = \{[A,f,g] : (A,f,g) \in \Span_{\C}(X,Y)\}$.


\begin{definicao}
    A \textit{partial morphism} from $X$ to $Y$ is a span $(A,f,g)$ where $f$ is a monomorphism.
\end{definicao}

Denote by $\Par_{\C}(X,Y)$ the full subcategory of $\Span_{\C}(X,Y)$ whose objects are partial morphisms, and by $\pa_{\C}(X,Y)$ the corresponding subclass of $\spa_{\C}(X,Y)$. \red{Observe that there is at most one morphism between any two objects in $\Par_{\C}(X,Y)$.}

\begin{proposicao}\label{partial-morphisms-in-set}
    Let $\C = \Set$ and $(A,f,g) \in \Par_{\C}(X,Y)$. Then the isomorphism class $[A,f,g]$ has exactly one representative $(B,\iota,h)$ where $B \subseteq X$ and $\iota$ is the respective inclusion.
\end{proposicao}

\begin{proof}
    Let $B = f(A) \subseteq X$ and $\iota$ be the inclusion of $B$ in $X$. Let $h : B \rightarrow Y$ be the map given by $h(x) = g(a)$, if $x = f(a)$ for some $a \in A$, which is \red{well defined}, since $f$ is a monomorphism in $\Set$.

    It is then easily checked that the map $\varphi : A \rightarrow B$ such that $\varphi(a) = f(a)$ for each $a \in A$ is an isomorphism of spans from $(A,f,g)$ to $(B,\iota,h)$, so that $[A,f,g] = [B,\iota,h]$. The uniqueness is also a simple verification.
\end{proof}

The following result is a slight modification and, at the same time, a generalization of \cite[Lemma 1.7]{geopactions}.

\begin{proposicao}[see Lemma 1.7 from \cite{geopactions}]\label{correspondencia-dadosdeacao-set}
Let $\C = \Set$ and $X\in\C$. There is a one-to-one correspondence between
\begin{enumerate}
    \item partial action data of $M$ on $X$;
    
    
    \item maps from $M$ to $\pa_{\C}(X,X)$.
\end{enumerate}
\end{proposicao}

\begin{proof}
    To each partial action datum $\{\alpha_m\}_{m \in M}$ we associate the map $\alpha : M \rightarrow \pa_{\C}(X,X)$ given by $\alpha(m) = [\dom\alpha_m,\iota_m,\alpha_m]$, where $\iota_m : \dom\alpha_m \rightarrow X$ is the inclusion of $\dom\alpha_m$ on $X$.

    This association has the following inverse: given $\alpha : M \rightarrow \pa_{\C}(X,X)$, for each $m \in M$ take $(\dom\alpha_m,\iota_m,\alpha_m)$ to be the unique representative for $\alpha(m)$ where $\dom\alpha_m \subseteq X$ and $\iota_m$ is the respective inclusion (see \cref{partial-morphisms-in-set}). We then associate $\alpha$ to the partial action datum $\{\alpha_m\}_{m \in M}$.
\end{proof}


\begin{rem}
In \cite[Lemma 1.7]{geopactions}, in the setting of partial group actions, the authors constructed a correspondence analogous to that of \cref{correspondencia-dadosdeacao-set}, but with $\Par_{\Set}(X,X)$ instead of $\pa_{\Set}(X,X)$. However, it fails to be surjective, as monomorphisms in $\Set$ are not necessarily inclusions of subsets.
\end{rem}

\begin{definicao}
    Let $\C$ be a category, $\D$ a subcategory of $\C$ and $X \in \C$. A \textit{reflection} of $X$ in $\D$ (or a \textit{$\D$-reflection} of $X$) is a morphism $r : X \to Y$ with $Y \in \D$ such that for any $f \in \Hom{\C}(X,Z)$ with $Z \in \D$ there is a unique $f' \in \Hom{\D}(Y,Z)$ such that the following diagram commutes. 
    \begin{equation*}
        \begin{tikzcd}
X \arrow[rd, "f"'] \arrow[r, "r"] & Y \arrow[d, "f'", dashed] \\
                                  & Z                        
\end{tikzcd}
    \end{equation*}

    In this situation, we say that $X$ has a reflection in $\D$.
\end{definicao}

\section{Partial monoid actions on objects in categories with pullbacks}\label{sec-pact}

\blue{Throughout this work we assume that $M$ is} a monoid and $\C$ a category with pullbacks, \blue{although everything works for a wider class of categories, admitting pullbacks of any two morphisms of which one is a \textit{monomorphism} (this more general case is only considered in \cref{exemplo-glob-sem-reflexao} below).}

\cref{correspondencia-dadosdeacao-set} motivates the following definition.

\begin{definicao}
A \textit{partial action datum} of $M$ on $X \in \C$ is a map $\alpha : M \rightarrow \pa_{\C}(X,X)$.
\end{definicao}


We can describe global actions on objects in $\C$ in terms of partial action data as follows:


\begin{definicao}
    A \textit{global action} of $M$ on $X \in \C$ is a partial action datum $\alpha$ of $M$ on $X$ such that:
    \begin{enumerate}[label=(CGA\arabic*)]
        \item\label{AG1} $\alpha(e) = [X,id_X,id_X]$;

        \item\label{AG2} \red{$\alpha(m) = [X,id_X,\alpha_m]$ with $\dom\alpha_m = X$ for all $m \in M$, and} $\alpha_n \circ \alpha_m = \alpha_{nm}$ for all $n,m \in M$.
        
    \end{enumerate}
\end{definicao}



In order to define a partial action of $M$ on $X\in \C$, for any partial action datum $\alpha(m) = [\dom\alpha_m,\iota_m,\alpha_m]$ of $M$ on $X$ we shall fix a pullback $\alpha_m^{-1}(\dom\alpha_n)$ of $\alpha_m$ and $\iota_n$ as illustrated in the diagram.

\begin{equation}\label{diagrama-imagem-inversa-dado-acao-parcial}
\begin{tikzcd}
                                     & \alpha_m^{-1}(\dom\alpha_n) \arrow[ld, "{\hat{\iota}_n^m}"'] \arrow[rd, "{\widehat{\alpha}_m^n}"] &                                    \\
\dom\alpha_m \arrow[rd, "\alpha_m"'] &                                                                                      & \dom\alpha_n \arrow[ld, "\iota_n"] \\
                                     & X                                                                                    &                                   
\end{tikzcd}
\end{equation}

\red{We point out that the notation for the morphisms $\hat{\iota}_n^m$ and $\widehat{\alpha}_m^n$ associated with $\iota_n$ and $\af_m$ as in \cref{diagrama-imagem-inversa-dado-acao-parcial} will be used in \cref{proposicao-restricao-eh-forte,acao parcial-forte,acao parcial,restricao-eh-acao-parcial}.}

\begin{definicao}\label{acao parcial}
    A \textit{partial action} of $M$ on $X \in \C$ is a partial action datum $\alpha(m) = [\dom\alpha_m,\iota_m,\alpha_m]$ of $M$ on $X$ such that:
    \begin{enumerate}[label=(CPA\arabic*)]
        \item\label{AP1} $\alpha(e) = [X,id_X,id_X]$;

        \item\label{AP2} for all $m, n \in M$ there is a morphism $\varphi :\alpha_m^{-1}(\dom\alpha_n) \to  \dom\alpha_{nm}$ such that the following diagram commutes.
        \begin{equation}\label{diagrama-acao-parcial}
\begin{tikzcd}
  &  & \alpha_m^{-1}(\dom\alpha_n) \arrow[rrdd, "{\alpha_{n} \circ \widehat{\alpha}_m^n}"] \arrow[d, "\varphi", dashed] \arrow[lldd, "{\iota_m \circ \hat{\iota}_n^m}"'] &  &   \\
  &  & \dom\alpha_{nm} \arrow[rrd, "\alpha_{nm}"'] \arrow[lld, "\iota_{nm}"]                                                                                  &  &   \\
X &  &                                                                                                                                                        &  & X
\end{tikzcd}
    \end{equation}
    \end{enumerate}
\end{definicao}


\begin{rem}
    The existence of the morphism $\varphi$ in \cref{acao parcial} does not depend on the choices of a pullback of $\alpha_m$ and $\iota_n$ and of a representative for the equivalence class $\alpha(m) = [\dom\alpha_m,\iota_m,\alpha_m]$. 
\end{rem}

\begin{rem}
    A partial action is a generalization of a global action of $M$ on $X$, because in the global case the diagram \cref{diagrama-acao-parcial} becomes
        \begin{equation*}
\begin{tikzcd}
  &  & X \arrow[rrdd, "\alpha_n \circ \alpha_m"] \arrow[d, "\varphi", dashed] \arrow[lldd, "id_X"'] &  &   \\
  &  & X \arrow[rrd, "\alpha_{nm}"'] \arrow[lld, "id_X"]                                            &  &   \\
X &  &                                                                                              &  & X
\end{tikzcd}
    \end{equation*}
with $\vf=id_X$, if we take $\hat{\iota}_n^m = id_X$ and $\widehat{\alpha}_m^n = \alpha_m$.
\end{rem}


\begin{rem}
    In the setting of \cref{acao parcial} take $\C = \Set$, $\dom\alpha_m \subseteq X$ with $\iota_m$ being the corresponding inclusion for all $m \in M$ and the usual inverse image in \cref{diagrama-imagem-inversa-dado-acao-parcial}. Then the commutativity of \cref{diagrama-acao-parcial} is equivalent to \cref{AP2 set,AP3 set}, with $\varphi$ being the corresponding inclusion. Thus, \cref{acao parcial} restricts to \cref{acao_parcial_em_conjuntos}.

\end{rem}

The concept of a morphism between partial action data can also be described in this context. 

\begin{definicao}\label{definicao-datum-morphism}
    Let $\alpha(m) = [\dom\alpha_m,\iota_m,\alpha_m]$ and $\beta(m) = [\dom\beta_m,\kappa_m,\beta_m]$ be partial action data on, respectively, objects $X$ and $Y$ in $\C$. A \textit{partial action datum morphism} (in short, \textit{datum morphism}) from $\alpha$ to $\beta$ is a morphism $f : X \rightarrow Y$ in $\C$ such that for all $m \in M$ there exists a morphism $f_m : \dom\alpha_m \to \dom\beta_m$ in $\C$ making the following diagram commute.
    \begin{equation}\label{diagrama-G-datum}
\begin{tikzcd}
                                     & \dom\alpha_m \arrow[ld, "\iota_m"'] \arrow[rd, "\alpha_m"] \arrow[dd, "f_m", dashed] &                   \\
X \arrow[dd, "f"]  &                                                                                    & X \arrow[dd, "f"] \\
                                     & \dom\beta_m \arrow[ld, "\kappa_m"'] \arrow[rd, "\beta_m"]                           &                   \\
Y                &                                                                                    & Y                
\end{tikzcd}
\end{equation}
\end{definicao}


It can be shown that the existence of $f_m$ in \cref{definicao-datum-morphism} does not depend on the choices of representatives of $\alpha(m)$ and $\beta(m)$.

We can then define the categories of partial action data, partial actions and global actions.

\begin{definicao}\label{categoria-acao-parcial}
    Denote by $\datum{\C}$ the category whose objects are partial action data of $M$ on objects in $\C$ and whose morphisms are the datum morphisms between those objects, where the composition is inherited from $\C$. Moreover, let $\pact{\C}$ (resp. $\act{\C}$) denote the full subcategory of $\datum{\C}$ formed by the partial (resp. global) actions of $M$ on objects in $\C$.
\end{definicao}





In order to avoid the recurrence to the diagram \cref{diagrama-G-datum} when dealing with datum morphisms that go to global actions, we have the following lemma.

\begin{lema}\label{morfismo-de-dados-para-global}
Let $\alpha(m) = [\dom\alpha_m,\iota_m,\alpha_m]$ be a partial action datum on $X \in \C$ and $\beta(m) = [Y,id_Y, \beta_m]$ be a global action of $M$ on $Y \in \C$. Then a morphism $f$ from $X$ to $Y$ in $\C$ is a datum morphism from $\alpha$ to $\beta$ if and only if the following diagram commutes for all $m \in M$.
    \begin{equation}\label{equacao-morfismo-de-dados-para-global}
\begin{tikzcd}
                                 & \dom\alpha_m \arrow[rd, "\alpha_m"] \arrow[ld, "\iota_m"'] &                   \\
X \arrow[rd, "\beta_m \circ f"'] &                                                            & X \arrow[ld, "f"] \\
                                 & Y                                                          &                  
\end{tikzcd}
    \end{equation}
 

\end{lema}

\begin{proof}
If $f$ is a datum morphism from $\alpha$ to
$\beta$, then for each $m \in M$ there exists a morphism $f_m : \dom\alpha_m \to Y$ such that the diagram
    \begin{equation}\label{diagrama-morfismo-de-dados-para-global}
\begin{tikzcd}
                  & \dom\alpha_m \arrow[ld, "\iota_m"'] \arrow[rd, "\alpha_m"] \arrow[dd, "f_m", dashed] &                   \\
X \arrow[dd, "f"] &                                                                                      & X \arrow[dd, "f"] \\
                  & Y \arrow[ld, "id"'] \arrow[rd, "\beta_m"]                                            &                   \\
Y                 &                                                                                      & Y                
\end{tikzcd}
    \end{equation}
commutes. The commutativity of the left square of \cref{diagrama-morfismo-de-dados-para-global} yields $f_m = f \circ \iota_m$, which together with the commutativity of the right square yields $\beta_m \circ f \circ \iota_m = \beta_m \circ f_m = f \circ \alpha_m$, as desired.

Conversely, if \cref{equacao-morfismo-de-dados-para-global} commutes for all $m \in M$, then the morphism $f_m \coloneqq f \circ \iota_m$ makes the diagram \cref{diagrama-morfismo-de-dados-para-global} commute, and hence $f$ is a datum morphism.
\end{proof}

\section{Restrictions of global actions}\label{sec-restr}

In this categorical context, we can also construct partial actions from global actions, similarly to \cref{exemplo global}.

In what follows in this section, whenever $\beta$ is a global action of $M$ on $Y \in \C$, we will assume that $\beta(m) = [Y,id_Y,\beta_m]$ for all $m \in M$.

\begin{definicao}\label{induced-partial-action-def}

    Let $\beta$ be a global action of $M$ on $Y \in \C$ and $\iota : X \rightarrow Y$ a monomorphism in $\C$. The \textit{restriction} of $\beta$ to $X$ (via $\iota$) is the partial action datum $\alpha(m) = [\dom\alpha_m,\iota_m,\alpha_m]$ \blue{of $M$ on $X$} such that the following diagram is a pullback.
    \begin{equation}\label{diagrama-pullback-acao-parcial-induzida}
        \begin{tikzcd}
                                        & \dom\alpha_m \arrow[rd, "\alpha_m"] \arrow[ld, "\iota_m"'] &                         \\
X \arrow[rd, "\beta_m \circ \iota"'] &                                        & X \arrow[ld, "\iota"] \\
                                        & Y                                      &                        
\end{tikzcd}
    \end{equation}
\end{definicao}


\begin{rem}\label{remark-induced-partial-action}
    The restriction $\alpha$ of $\beta$ to $X$ in \cref{induced-partial-action-def} exists as a partial action datum (that is, $\iota_m$ is a monomorphism for all $m \in M$) because pullbacks preserve monomorphisms, and $\alpha$ does not depend on the choice of a pullback of $\beta_m \circ \iota$ and $\iota$, since any two such pullbacks are always isomorphic spans.
\end{rem}

\begin{rem}\label{remark-induced-partial-action-morfismo-eh-datum-morphism}
    Notice that by \cref{morfismo-de-dados-para-global} the morphism $\iota : X \to Y$ in \cref{induced-partial-action-def} is a datum morphism from $\alpha$ to $\beta$.
\end{rem}


\begin{proposicao}

    Let $Y$ be a set, $\beta$ a global action of $M$ on $Y$ and $X \subseteq Y$ with $\iota$ being the inclusion of $X$ into $Y$. Then the restriction of $\beta$ to $X$ in the sense of \cref{restriction-definicao-set} corresponds (as in \cref{correspondencia-dadosdeacao-set}) to the restriction of $\beta$ to $Y$ via $\iota$ in the sense of \cref{induced-partial-action-def}.
\end{proposicao}

\begin{proof}
    The restriction of $\beta$ to $X$ in the sense of \cref{restriction-definicao-set} is the partial action $\{\alpha_m\}_{m \in M}$ where, for each $m \in M$
    $$\dom\alpha_m = X \cap \beta_m^{-1}(X)$$
    and $\alpha_m : \dom\alpha_m \to X$ is given by 
    \begin{equation}\label{alphag(a)=betag(a)}
        \alpha_m(x) = \beta_m(x)
    \end{equation}
    for each $x \in \dom\alpha_m$.

    By \cref{correspondencia-dadosdeacao-set} the family $\{\alpha_m\}_{m \in M}$ corresponds to $\alpha(m) = [\dom\alpha_m,\iota_m,\alpha_m]$, where $\iota_m$ is the inclusion of $\dom\alpha_m$ into $X$. 

    Notice that the diagram \cref{diagrama-pullback-acao-parcial-induzida} commutes for all $m \in M$, by \cref{alphag(a)=betag(a)}. \red{It is then a simple exercise to verify that \cref{diagrama-pullback-acao-parcial-induzida} is a pullback for all $m \in M$, so $\alpha$ is the restriction of $\beta$ to $X$ via $\iota$, as desired.}

\end{proof}


We shall prove below that any restriction of a global action is a partial action. To this end, for the remainder of this section, assume that we are in the setting of \cref{induced-partial-action-def}.

\begin{proposicao}\label{restricao-eh-acao-parcial}
    The restriction $\alpha$ of $\beta$ to $X$ in \cref{induced-partial-action-def} is a partial action of $M$ on $X$.
\end{proposicao}

\begin{proof}
    We first check \cref{AP1}. Note that $\beta_e = \text{id}_Y$, so $\beta_e \circ \iota = \iota$. Since $\iota$ is a monomorphism, the following diagram is a pullback.
    \begin{center}
    \begin{tikzcd}
                                  & X \arrow[rd, "id_X"] \arrow[ld, "id_X"'] &                   \\
X \arrow[rd, "\beta_e \circ \iota"'] &                                          & X \arrow[ld, "\iota"] \\
                                  & Y                                        &                  
\end{tikzcd}
    \end{center}
Therefore, by \cref{remark-induced-partial-action},
    $$\alpha(e) = [X,id_X,id_X].$$

To verify \cref{AP2}, fix $m,n \in M$. Our goal is to construct $\varphi: \alpha_m^{-1}(\dom\alpha_n) \to \dom\alpha_{nm}$ which makes diagram \cref{diagrama-acao-parcial} commute.

Notice that, since $\beta$ is a global action, and by the commutativity of the diagrams \cref{diagrama-imagem-inversa-dado-acao-parcial,diagrama-pullback-acao-parcial-induzida},
\begin{align*}
    (\beta_{nm} \circ \iota) \circ (\iota_m \circ \hat{\iota}_n^m) &= (\beta_n \circ \beta_m) \circ \iota \circ \iota_m \circ \hat{\iota}_n^m = \beta_n \circ (\beta_m \circ \iota \circ \iota_m) \circ \hat{\iota}_n^m \\&
    = \beta_n \circ (\iota \circ \alpha_m) \circ \hat{\iota}_n^m = \beta_n \circ \iota \circ (\alpha_m \circ \hat{\iota}_n^m) = \beta_n \circ \iota \circ (\iota_n \circ \widehat{\alpha}_n^m) \\
    &= (\beta_n \circ \iota \circ \iota_n) \circ \widehat{\alpha}_n^m = (\iota \circ \alpha_n) \circ \widehat{\alpha}_n^m = \iota \circ (\alpha_n \circ \widehat{\alpha}_n^m).
\end{align*}

Thus, the diagram
\begin{equation*}
    \begin{tikzcd}
                                        & \alpha_m^{-1}(\dom\alpha_n) \arrow[ld, "\iota_m \circ \hat{\iota}_n^m"'] \arrow[rd, "\alpha_n \circ \widehat{\alpha}_n^m"] &                       \\
X \arrow[rd, "\beta_{nm} \circ \iota"'] &                                                                                                                            & X \arrow[ld, "\iota"] \\
                                        & Y                                                                                                                          &                      
\end{tikzcd}
\end{equation*}
commutes, and so, by the universal property of the pullback \cref{diagrama-pullback-acao-parcial-induzida}, there exists a unique morphism $\varphi$ such that the diagram
\begin{equation*}
    \begin{tikzcd}
                                        & \alpha_m^{-1}(\dom\alpha_n) \arrow[ldd, "\iota_m \circ \hat{\iota}_n^m"'] \arrow[rdd, "\alpha_n \circ \widehat{\alpha}_n^m"] \arrow[d, "\varphi", dashed] &                       \\
                                        & \dom\alpha_{nm} \arrow[ld, "\iota_{nm}"] \arrow[rd, "\alpha_{nm}"']                                                                                       &                       \\
X \arrow[rd, "\beta_{nm} \circ \iota"'] &                                                                                                                                                           & X \arrow[ld, "\iota"] \\
                                        & Y                                                                                                                                                         &                      
\end{tikzcd}
\end{equation*}
commutes, and so $\varphi$ makes \cref{diagrama-acao-parcial} commute, as desired.

\end{proof}


\red{In \cref{proposicao-restricao-eh-forte,acao parcial-forte},} for any partial action datum $\alpha(m) = [\dom\alpha_m,\iota_m,\alpha_m]$ of $M$ on $X$ we shall fix a pullback $\dom\alpha_m \cap \dom\alpha_n$ of $\iota_m$ and $\iota_n$ as illustrated in the diagram.

\begin{equation}\label{diagrama-interseccao-dado-acao-parcial}
\begin{tikzcd}
                                    & \dom\alpha_m \cap \dom\alpha_n \arrow[ld, "\overline{\iota}_n^m"'] \arrow[rd, "\overline{\iota}_m^n"] &                                    \\
\dom\alpha_m \arrow[rd, "\iota_m"'] &                                                                                                       & \dom\alpha_n \arrow[ld, "\iota_n"] \\
                                    & X                                                                                                     &                                   
\end{tikzcd}
\end{equation}

The restriction of a global action moreover satisfies the following property.

\begin{proposicao}\label{proposicao-restricao-eh-forte}
    Let $\alpha$ be the restriction of $\beta$ to $X$ in \cref{induced-partial-action-def}. Then for each $m,n \in M$ there exists an isomorphism $\theta : \alpha_m^{-1}(\dom\alpha_n) \to \dom\alpha_m \cap \dom\alpha_{nm}$ such that the diagram
    \begin{equation}\label{diagrama-proposicao-restricao-eh-forte}
\begin{tikzcd}
  &  & \alpha_m^{-1}(\dom\alpha_n) \arrow[d, "\theta", dashed] \arrow[rrdd, "\alpha_{n} \circ \widehat{\alpha}_m^n"] \arrow[lldd, "\iota_m \circ \hat{\iota}_n^m"'] &  &   \\
  &  & \dom\alpha_m \cap \dom\alpha_{nm} \arrow[rrd, "\alpha_{nm} \circ \overline{\iota}_m^{nm}"'] \arrow[lld, "\iota_{nm} \circ \overline{\iota}_m^{nm}"]       &  &   \\
X &  &                                                                                                                                                           &  & X
\end{tikzcd}
    \end{equation}
    commutes.
\end{proposicao}

\begin{proof}
By \cref{restricao-eh-acao-parcial}, $\alpha$ is a partial action. Let $\varphi : \alpha_m^{-1}(\dom\alpha_n) \to \dom\alpha_{nm}$ be the morphism from \cref{AP2} which makes the diagram \cref{diagrama-acao-parcial} commute.

Then notice that, by the commutativity of \cref{diagrama-acao-parcial},
$$\iota_{nm} \circ \varphi = \iota_m \circ \hat{\iota}_n^m,$$
so since \cref{diagrama-interseccao-dado-acao-parcial} is a pullback there is a unique morphism $\theta : \alpha_m^{-1}(\dom\alpha_n) \to \dom\alpha_m \cap \dom\alpha_{nm}$ that makes the following diagram commute.
\begin{equation}\label{diagrama-proposicao-restricao-eh-forte-nu}
\begin{tikzcd}
                                    & \alpha_m^{-1}(\dom\alpha_n) \arrow[ldd, "\hat{\iota}_n^m"'] \arrow[rdd, "\varphi"] \arrow[d, "\theta", dashed]    &                                       \\
                                    & \dom\alpha_m \cap \dom\alpha_{nm} \arrow[ld, "\overline{\iota}_{nm}^m"] \arrow[rd, "\overline{\iota}_m^{nm}"'] &                                       \\
\dom\alpha_m \arrow[rd, "\iota_m"'] &                                                                                                                & \dom\alpha_{nm} \arrow[ld, "\iota_{nm}"] \\
                                    & X                                                                                                              &                                      
\end{tikzcd}
\end{equation}

The morphism $\theta$ makes \cref{diagrama-proposicao-restricao-eh-forte} commute, because
$$(\iota_{nm} \circ \overline{\iota}_m^{nm}) \circ \theta = (\iota_m \circ \overline{\iota}_{nm}^m) \circ \theta = \iota_m \circ (\overline{\iota}_{nm}^m \circ \theta) = \iota_m \circ \hat{\iota}_n^m$$
by the commutativity of \cref{diagrama-proposicao-restricao-eh-forte-nu,diagrama-interseccao-dado-acao-parcial} and 
$$(\alpha_{nm} \circ \overline{\iota}_m^{nm}) \circ \theta = \alpha_{nm} \circ (\overline{\iota}_m^{nm} \circ \theta) = \alpha_{nm} \circ \varphi = \alpha_n \circ \widehat{\alpha}_m^n$$
by the commutativity of \cref{diagrama-proposicao-restricao-eh-forte-nu,diagrama-acao-parcial}.

Let us verify that $\theta$ is an isomorphism by exhibiting its inverse. Notice that by the commutativity of \cref{diagrama-pullback-acao-parcial-induzida,diagrama-interseccao-dado-acao-parcial} we have
    \begin{align*}
        \iota \circ (\alpha_{nm} \circ \overline{\iota}_m^{nm}) &= (\iota \circ \alpha_{nm}) \circ \overline{\iota}_m^{nm} = (\beta_{nm} \circ \iota \circ \iota_{nm}) \circ \overline{\iota}_m^{nm} \\
        &= (\beta_n \circ \beta_m) \circ \iota \circ (\iota_{nm} \circ \overline{\iota}_m^{nm}) = \beta_n \circ \beta_m \circ \iota \circ (\iota_m \circ \overline{\iota}_{nm}^m) \\
        &= \beta_n \circ (\beta_m \circ \iota \circ \iota_m) \circ \overline{\iota}_{nm}^m = \beta_n \circ (\iota \circ \alpha_m) \circ \overline{\iota}_{nm}^m \\
        &= \beta_n \circ \iota \circ (\alpha_m \circ \overline{\iota}_{nm}^m).
    \end{align*}
so since \cref{diagrama-pullback-acao-parcial-induzida} is a pullback, there exists a unique morphism $\eta : \dom\alpha_m \cap \dom\alpha_{nm} \to \alpha_m^{-1}(\dom\alpha_n)$ such that the following diagram commutes.
\begin{equation}\label{diagrama-proposicao-restricao-eh-forte-eta}
\begin{tikzcd}
                                     & \dom\alpha_m \cap \dom\alpha_{nm} \arrow[ldd, "\alpha_m \circ \overline{\iota}_{nm}^m"'] \arrow[rdd, "\alpha_{nm} \circ \overline{\iota}_m^{nm}"] \arrow[d, "\eta", dashed] &                       \\
                                     & \dom\alpha_n \arrow[ld, "\iota_n"] \arrow[rd, "\alpha_n"']                                                                                                                  &                       \\
X \arrow[rd, "\beta_n \circ \iota"'] &                                                                                                                                                                             & X \arrow[ld, "\iota"] \\
                                     & Y                                                                                                                                                                           &                      
\end{tikzcd}
\end{equation}

In particular, by the commutativity of \cref{diagrama-proposicao-restricao-eh-forte-eta} it follows that
\begin{align*}
    \iota_n \circ \eta = \alpha_m \circ \overline{\iota}_{nm}^m,
\end{align*}
so since \cref{diagrama-imagem-inversa-dado-acao-parcial} is a pullback, there exists a unique morphism $\psi : \dom\alpha_m \cap \dom\alpha_{nm} \to \alpha_m^{-1}(\dom\alpha_n)$ such that the following diagram commutes.
\begin{equation}\label{diagrama-proposicao-restricao-eh-forte-psi}
    \begin{tikzcd}
                                     & \dom\alpha_m \cap \dom\alpha_{nm} \arrow[ldd, "\overline{\iota}_{nm}^m"'] \arrow[rdd, "\eta"] \arrow[d, "\psi", dashed] &                                    \\
                                     & \alpha_m^{-1}(\dom\alpha_n) \arrow[ld, "\hat{\iota}_n^m"] \arrow[rd, "\widehat{\alpha}_m^n"']                              &                                    \\
\dom\alpha_m \arrow[rd, "\alpha_m"'] &                                                                                                                            & \dom\alpha_n \arrow[ld, "\iota_n"] \\
                                     & X                                                                                                                          &                                   
\end{tikzcd}
\end{equation}

We shall verify that $\psi$ is the inverse of $\theta$. Notice that by the commutativity of \cref{diagrama-proposicao-restricao-eh-forte-psi,diagrama-proposicao-restricao-eh-forte-nu} we have
$$\overline{\iota}_{nm}^m \circ \theta \circ \psi = \hat{\iota}_n^m \circ \psi = \overline{\iota}_{nm}^m = \overline{\iota}_{nm}^m \circ id_{\dom\alpha_m \cap \dom\alpha_{nm}},$$
so, since $\overline{\iota}_{nm}^m$ is a monomorphism (because \cref{diagrama-interseccao-dado-acao-parcial} is a pullback for all $m, n \in M$ and $\iota_{nm}$ is a monomorphism), we have $\theta \circ \psi = id_{\dom\alpha_m \cap \dom\alpha_{nm}}$.

Similarly, the commutativity of \cref{diagrama-proposicao-restricao-eh-forte-psi,diagrama-proposicao-restricao-eh-forte-nu} gives us
$$\hat{\iota}_n^m \circ \psi \circ \theta = \overline{\iota}_{nm}^m \circ \theta = \hat{\iota}_n^m = \hat{\iota}_n^m \circ id_{\alpha_m^{-1}(\dom\alpha_n)},$$
so, since $\hat{\iota}_n^m$ is a monomorphism (because $\iota_n$ is a monomorphism in \cref{diagrama-imagem-inversa-dado-acao-parcial}), we have $\psi \circ \theta = id_{\alpha_m^{-1}(\dom\alpha_n)}$. Thus, $\psi$ is an isomorphism, as desired.
\end{proof}

\cref{proposicao-restricao-eh-forte} motivates the following definition.

\begin{definicao}\label{acao parcial-forte}
    A \textit{strong partial action} of $M$ on $X \in \C$ is a partial action datum $\alpha(m) = [\dom\alpha_m,\iota_m,\alpha_m]$ of $M$ on $X$ such that:
    \begin{enumerate}[label=(SCPA\arabic*)]
        \item\label{AP1strong} $\alpha(e) = [X,id_X,id_X]$;

        \item\label{AP2strong} for all $m, n \in G$ there is an isomorphism $\theta : \alpha_m^{-1}(\dom\alpha_n) \to \dom\alpha_m \cap \dom\alpha_{nm}$ such that the following diagram commutes.
        \begin{equation}\label{diagrama-acao-parcial-forte}
\begin{tikzcd}
  &  & \alpha_m^{-1}(\dom\alpha_n) \arrow[d, "\theta", dashed] \arrow[rrdd, "\alpha_{n} \circ \widehat{\alpha}_m^n"] \arrow[lldd, "\iota_m \circ \hat{\iota}_n^m"'] &  &   \\
  &  & \dom\alpha_m \cap \dom\alpha_{nm} \arrow[rrd, "\alpha_{nm} \circ \overline{\iota}_m^{nm}"'] \arrow[lld, "\iota_{nm} \circ \overline{\iota}_m^{nm}"]       &  &   \\
X &  &                                                                                                                                                           &  & X
\end{tikzcd}
    \end{equation}
    \end{enumerate}
\end{definicao}

\begin{rem}\label{remark-restricao-eh-acao-parcial-forte}
    \cref{proposicao-restricao-eh-forte,restricao-eh-acao-parcial} show that the restriction of a global action is a strong partial action.
\end{rem}


\begin{rem}
    Every strong partial action is a partial action. Indeed, if $\theta$ makes \cref{diagrama-acao-parcial-forte} commute, then $\varphi = \overline{\iota}_m^{nm} \circ \theta$ makes \cref{diagrama-acao-parcial} commute.
\end{rem}

\begin{definicao}
    Denote by $\spact{\C}$ the full subcategory of $\datum{\C}$ whose objects are the strong partial actions of $M$ on objects in $\C$.
\end{definicao}

\begin{rem}\label{remark-strong-partial-action-set}
    In the setting of \cref{acao parcial-forte} take $\C = \Set$, $\dom\alpha_m \subseteq X$ with $\iota_m$ being the corresponding inclusion for all $m \in M$ and the usual intersection and inverse image in \cref{diagrama-interseccao-dado-acao-parcial,diagrama-imagem-inversa-dado-acao-parcial}. Then the commutativity of \cref{diagrama-acao-parcial-forte} is equivalent to $\theta = id$, \cref{AP3 set,AP2 set strong}. Thus, \cref{acao parcial-forte} restricts to \cref{acao_parcial_forte_em_conjuntos}.
    

\end{rem}


\begin{exemplo}[Example 1.3 from \cite{geopactions}]\label{exm-non-strong-pa}
    Let $X = \Z$ and $M = (\Z,+)$.
\begin{align*}
    \dom\alpha_z =\begin{cases}
         \Z, & \text{ if } z \geq 0,\\
         \{x \in \Z : x \geq -z\}, & \text{ if } z < 0,
    \end{cases}
    \quad \text{ and } \quad \alpha_z(x) = x+z
\end{align*}
  for all $m \in M$. Then the partial action datum $\alpha(m) = [\dom\alpha_m,\iota_m,\alpha_m]$ of $M$ on $X$, where $\iota_z$ is the inclusion of $\dom\alpha_z$ into $\Z$, is a partial action of $M$ on $X$ that is not strong.
\end{exemplo}

The following example is an adaptation of \cite[Example 3.1]{saracco2022globalization2}.
\begin{exemplo}\label{exemplo-esp-top-acao-parcial-forte}
    Let $\C = \Top$, $M$ any nontrivial monoid, $X$ a set and $\tau$ and $\tau'$ topologies on $X$ such that $\tau'$ is a strictly finer than $\tau$. We have a strong partial action $\alpha$ of $M$ on $(X,\tau) \in \Top$ given as follows: $\alpha(e) = [(X,\tau),id_X,id_X]$ and $\alpha(m) = [(X,\tau'),id_X,id_X]$ if $m \neq e$.

    However, $\alpha$ is not a restriction of any global action. Indeed, assume by contradiction that $\alpha$ is the restriction of a global action $\beta$ on $(Y,\upsilon) \in \Top$ via some monomorphism $\iota : (X,\tau) \to (Y,\upsilon)$. Then, given any nontrivial $m \in M$, the diagram
    \begin{equation}\label{diagrama-pullback-exemplo-esp-top-acao-parcial-forte}
            \begin{tikzcd}
                                              & {(X,\tau')} \arrow[ld, "id_X"'] \arrow[rd, "id_X"] &                                \\
{(X,\tau)} \arrow[rd, "\beta_m \circ \iota"'] &                                                    & {(X,\tau)} \arrow[ld, "\iota"] \\
                                              & {(Y,\upsilon)}                                                  &                               
\end{tikzcd}
    \end{equation}
    is a pullback, and in particular commutative, so $\beta_m \circ \iota = \iota$. Thus, the diagram
    \begin{center}
        \begin{tikzcd}
                                              & {(X,\tau)} \arrow[ld, "id_X"'] \arrow[rd, "id_X"] &                                \\
{(X,\tau)} \arrow[rd, "\beta_m \circ \iota"'] &                                                    & {(X,\tau)} \arrow[ld, "\iota"] \\
                                              & {(Y,\upsilon)}                                                  &                               
\end{tikzcd}
    \end{center}
    is a commutative diagram in $\Top$. Since \cref{diagrama-pullback-exemplo-esp-top-acao-parcial-forte} is a pullback, there is a unique morphism $\varphi : (X,\tau) \to (X,\tau')$ such that the diagram
    \begin{center}
        \begin{tikzcd}
                                              & {(X,\tau)} \arrow[ldd, "id_X"'] \arrow[rdd, "id_X"] \arrow[d, "\varphi", dashed] &                                \\
                                              & {(X,\tau')} \arrow[ld, "id_X"] \arrow[rd, "id_X"']                               &                                \\
{(X,\tau)} \arrow[rd, "\beta_m \circ \iota"'] &                                                                                  & {(X,\tau)} \arrow[ld, "\iota"] \\
                                              & {(Y,\upsilon)}                                                                                &                               
\end{tikzcd}
    \end{center}
    commutes. Clearly, $\varphi$ must be the identity map on $X$, so it follows that $\tau$ is a finer topology than $\tau'$, a contradiction. 
\end{exemplo}

\section{Globalizations of partial actions}\label{sec-glob}


We can now define the notion which is in some sense inverse to the restriction of a global action.

\begin{definicao}\label{globalizacao-def}
    Let $\alpha$ be a partial action datum of $M$ on $X \in \C$. A \textit{globalization} of $\alpha$ is a pair $(\beta,\iota)$ formed by a global action $\beta$ of $M$ on an object $Y \in \C$ and a monomorphism $\iota : X \rightarrow Y$, such that $\alpha$ is the restriction of $\beta$ to $X$ via $\iota$.
\end{definicao}

By \cref{remark-induced-partial-action-morfismo-eh-datum-morphism}, the morphism $\iota$ in \cref{globalizacao-def} is a datum morphism from $\alpha$ to $\beta$.

\begin{definicao}\label{definicao-globalizacao-universal}
    Let $\alpha$ be a partial action datum of $M$ on $X \in \C$. A \textit{universal globalization} of $\alpha$ is a pair $(\beta,\iota)$ such that:
    
    \begin{enumerate}[label=(UG\arabic*)]
        \item\label{GU1} $(\beta, \iota)$ is a globalization of $\alpha$;

        \item\label{GU2} whenever $(\gamma,\kappa)$ is a globalization of $\alpha$, there exists a unique morphism $\kappa'$ such that the following diagram commutes.
        \begin{equation}\label{diagrama-GU2}
            \begin{tikzcd}
\alpha \arrow[r, "\iota"] \arrow[rd, "\kappa"'] & \beta \arrow[d, "\kappa'", dashed, pos = 0.2] \\
                                          & \gamma                        
\end{tikzcd}
        \end{equation}
        
    \end{enumerate}

\end{definicao}

\begin{rem}
    Observe that our concept of a universal globalization slightly differs from that of a globalization defined in \cite{saracco2022globalization} because we do not require the datum morphism $\iota$ in \cref{definicao-globalizacao-universal} to be a reflection of $\alpha$ in $\act{\C}$.

\end{rem}

Nevertheless, whenever a reflection $\iota$ of $\alpha$ in $\act{\C}$ exists, it gives us a necessary and sufficient condition for $\alpha$ to have a (universal) globalization.

\begin{teorema}\label{equivalencia-glob-universal-glob}
    Let $\alpha(m) = [\dom\alpha_m,\iota_m,\alpha_m]$ be a partial action datum of $M$ on $X \in \C$. Assume that $\alpha$ has a reflection $\iota : \alpha \rightarrow \beta$ in $\act{\C}$, with, say, $\beta$ acting on $Y \in \C$. Then the following are equivalent:

    \begin{enumerate}
        \item\label{glob-universal-glob-i} $(\beta,\iota)$ is a globalization of $\alpha$;
        \item\label{glob-universal-glob-ii} $(\beta,\iota)$ is a universal globalization of $\alpha$;
        \item\label{glob-universal-glob-iii} $\alpha$ has a universal globalization;
        \item\label{glob-universal-glob-iv} $\alpha$ has a (not necessarily universal) globalization;
        \item\label{glob-universal-glob-v} for all $m \in M$ the following diagram is a pullback diagram in $\C$.
        \begin{equation}\label{diagrama-equivalencia-glob-universal-glob}
              \begin{tikzcd}[ampersand replacement=\&, row sep = small]
                             \& \dom\alpha_m \arrow[ld, "\iota_m"'] \arrow[rd, "\alpha_m"] \&                                    \\
X \arrow[rd, "\beta_m \circ \iota"'] \&                                                     \& X \arrow[ld, "\iota"] \\
                             \& Y                                                   \&                                   
\end{tikzcd}
        \end{equation}
    \end{enumerate}
\end{teorema}

\begin{proof}

    Implication $\cref{glob-universal-glob-i} \Rightarrow \cref{glob-universal-glob-ii}$ follows because $(\beta,\iota)$, being a globalization, satisfies \cref{GU1}, and \cref{GU2} is a consequence of the fact that $\iota$ is a reflection.

    Implications $\cref{glob-universal-glob-ii} \Rightarrow \cref{glob-universal-glob-iii}$ and $\cref{glob-universal-glob-iii} \Rightarrow \cref{glob-universal-glob-iv}$ are immediate and $\cref{glob-universal-glob-v} \Rightarrow \cref{glob-universal-glob-i}$ follows from the definition of a globalization, so it remains to check $\cref{glob-universal-glob-iv} \Rightarrow \cref{glob-universal-glob-v}$. 
    
    Assume thus that $\alpha$ has a globalization $(\gamma,\kappa)$, with, say, $\gamma$ acting on $Z \in \C$. By \cref{globalizacao-def}, $\alpha$ is the restriction of $\gamma$ to $X$ via $\kappa$. That is, for all $m \in M$ the diagram
    \begin{center}
              \begin{tikzcd}[ampersand replacement=\&, row sep = small]
                             \& \dom\alpha_m \arrow[ld, "\iota_m"'] \arrow[rd, "\alpha_m"] \&                                    \\
X \arrow[rd, "\gamma_m \circ \kappa"'] \&                                                     \& X \arrow[ld, "\kappa"] \\
                             \& Z                                                   \&                                   
\end{tikzcd}
        \end{center}
    is a pullback diagram.
    
    Since $\iota$ is a reflection of $\alpha$ in $\act{\C}$ and $\gamma$ is a global action with $\kappa$ a datum morphism from $\alpha$ to $\gamma$, there exists a unique datum morphism $\kappa' : \beta \rightarrow \gamma$ such that 
    \begin{equation}\label{kappa=kappa'-circ-iota}
        \kappa = \kappa' \circ \iota.
    \end{equation}
    Fix $m \in M$. As $\kappa'$ and $\iota$ are datum morphisms, by \cref{morfismo-de-dados-para-global} we have
    \begin{align}\label{gammag-circ-kappa'=-e-beta-g-circ-iota-circ-iota-g}
        \gamma_m \circ \kappa' = \kappa' \circ \beta_m \text{ and }
        \beta_m \circ \iota \circ \iota_m = \iota \circ \alpha_m.
    \end{align}
    This way, by \cref{kappa=kappa'-circ-iota,gammag-circ-kappa'=-e-beta-g-circ-iota-circ-iota-g}, the diagram 
        \begin{center}
        \begin{tikzcd}
                                                                           & \dom\alpha_m \arrow[ld, "\iota_m"'] \arrow[rd, "\alpha_m"] &                                              \\
X \arrow[rd, "\beta_m \circ \iota"] \arrow[rdd, "\gamma_m \circ \kappa"'] &                                                         & X \arrow[ld, "\iota"'] \arrow[ldd, "\kappa"] \\
                                                                           & Y \arrow[d, "\kappa'" pos = 0.2]                                  &                                              \\
                                                                           & Z                                                       &                                             
\end{tikzcd}
    \end{center}
    commutes. Now, since its perimeter is a pullback diagram, it is customary to check that the inner square is also a pullback diagram.
\end{proof}

\begin{rem}\label{diagrama-equivalencia-glob-universal-glob-comuta}
    Since $\iota$ is a datum morphism from $\alpha$ to $\beta$, diagram \cref{diagrama-equivalencia-glob-universal-glob} is already a commutative diagram, by \cref{morfismo-de-dados-para-global}.
\end{rem}

The following example shows that there may exist a universal globalization $(\beta,\iota)$ whose $\iota$ is not a reflection.

\begin{exemplo}\label{exemplo-glob-sem-reflexao}
    Let $\K$ be a field and $\C$ be the category of (associative\red{, not necessarily unital}) $\K$-algebras whose morphisms $X\to Y$ are $\K$-algebra homomorphisms $\varphi: X\to Y$ such that $\varphi(X)$ is an ideal in $Y$. Observe that $\C$ is a category \blue{where every morphism has a pullback with any monomorphism, and those are} inherited from $\Set$. Let $M = \Z_2 = \{0,1\}$ seen as a monoid under addition modulo $2$, and $X$ any unital non-zero $\K$-algebra.


    Consider the global action $\beta$ of $M$ on $X \times X$ given by $\beta_1(x,y) = (y,x)$, and the monomorphism $\iota : X \to X \times X$ with $\iota(x) = (x,0)$. Let $\alpha$ be the induced partial action of $\beta$ on $X$ via $\iota$, i.e. $\alpha(1) = [\{0\},\iota_1,\alpha_1]$. Notice that any element $(x,y) \in X \times X$ can be written as
    \begin{equation*}
        (x,y) = \iota(x) + \beta_1(\iota(y)).
    \end{equation*}
    
    We will verify that $(\beta,\iota)$ is a universal globalization of $\alpha$, but $\iota$ is not a reflection of $\alpha$ in $\act{\C}$. Since the pair $(\beta,\iota)$ satisfies \cref{GU1} by definition of $\alpha$, it suffices to check \cref{GU2}.
    
    Let $(\gamma,\kappa)$ be a globalization of $\alpha$, where, say, $\gamma$ is an action of $M$ on an algebra $Y$. Consider the map $\kappa' : X \times X \to Y$ given by 
    \begin{equation}\label{definicao-kappa'-exemplo-algebra}
        \kappa'(x,y) = \kappa(x) + \gamma_1(\kappa(y)).
    \end{equation}
    Then
    $$\kappa'(\iota(x)) = \kappa'(x,0) = \kappa(x),$$
    so $\kappa' \circ \iota = \kappa$. It is easy to see that $\kappa'$ is $\K$-linear and $\kappa'(X\times X)$ is an ideal of $Y$ (since $\kappa(X)$ is an ideal in $Y$ and $\gamma_1 \in \Aut(Y)$). Let us check that $\kappa'$ is multiplicative and that it is a datum morphism from $\beta$ to $\gamma$.
    
    Since $(\gamma,\kappa)$ is a globalization of $\beta$, the diagram
    \begin{center}
        \begin{tikzcd}
                                       & \{0\} \arrow[ld] \arrow[rd] &                        \\
X \arrow[rd, "\gamma_1 \circ \kappa"'] &                             & X \arrow[ld, "\kappa"] \\
                                       & Y                           &                       
\end{tikzcd}
    \end{center}
    is a pullback diagram. Thus, it follows that 
    \begin{equation}\label{gamma-m-kappa-cap-kappa=0}
        \gamma_1(\kappa(X)) \cap \kappa(X) = \{0\}.
    \end{equation} 

    Now, let $(x,y), (x',y') \in X \times X$. Then
    \begin{align}
        \kappa'(x,y) \kappa'(x',y') &= (\kappa(x) + \gamma_1(\kappa'(y)))(\kappa(x') + \gamma_1(\kappa(y'))) \notag\\
        &= \kappa(x)\kappa(x') + \gamma_1(\kappa(y)) \gamma_1(\kappa(y'))\label{kappa'-da-multiplicacao}\\
        &\quad+ \kappa(x) \gamma_1(\kappa(y')) + \gamma_1(\kappa(y)) \kappa(x').\label{k(x)gm(k(y'))+gm_1(k(y))k(x')}
    \end{align}
    Since $\kappa(X)$ and $\gamma_1(\kappa(X))$ are ideals in $Y$, both of the summands of \cref{k(x)gm(k(y'))+gm_1(k(y))k(x')} belong to $\gamma_1(\kappa(X)) \cap \kappa(X)$. Thus, by \cref{gamma-m-kappa-cap-kappa=0} the sum \cref{k(x)gm(k(y'))+gm_1(k(y))k(x')} is zero, so by \cref{kappa'-da-multiplicacao,definicao-kappa'-exemplo-algebra} it follows that
    \begin{align*}
        \kappa'(x,y) \kappa'(x',y') &= \kappa(x)\kappa(x') + \gamma_1(\kappa(y)) \gamma_1(\kappa(y')) = \kappa(x x') + \gamma_1(\kappa(y y')) \\
        &= \kappa'(x x',y y') = \kappa'((x,y)(x',y')),
    \end{align*}
    so $\kappa'$ is a morphism in $\C$.

    Moreover, by \cref{morfismo-de-dados-para-global}, $\kappa'$ is a datum morphism from $\beta$ to $\gamma$, since
    $$\kappa'(\beta_m(x,y)) = \kappa'(y,x) = \kappa(y) + \gamma_1(\kappa(x)) = \gamma_1(\gamma_1(\kappa(y)) + \kappa(x)) = \gamma_1(\kappa'(x,y)).$$

    So, there exists a morphism $\kappa' : \beta \to \gamma$ such that diagram \cref{diagrama-GU2} commutes. To check that it is unique, let $\kappa'' : \beta \to \gamma$ be a morphism such that $\kappa'' \circ \iota = \kappa$. Then by \cref{morfismo-de-dados-para-global} we have $\kappa'' \circ \beta_1 = \gamma_1 \circ \kappa''$, and it follows that
    \begin{align*}
        \kappa''(x,y) &= \kappa''((x,0) + \beta_m(y,0)) = \kappa''(\iota(x)) + \kappa''(\beta_m(\iota(y))) \\
        &= \kappa(x) + \gamma_1(\kappa''(\iota(y))) = \kappa(x) + \gamma_1(\kappa(y)) = \kappa'(x,y),
    \end{align*}
    so $\kappa'' = \kappa'$. Therefore, $(\beta,\iota)$ satisfies \cref{GU2}, and is, thus, a universal globalization of $\alpha$.

    However, $\iota$ is not a reflection of $\alpha$ in $\act{\C}$. To verify that, consider the global action $\gamma$ of $M$ on $X$ given by 
    \begin{equation*}
        \gamma_1 = id_X
    \end{equation*}
    and the morphism
    \begin{equation*}
        \kappa = id_X.
    \end{equation*}
    Notice that, by \cref{morfismo-de-dados-para-global}, $\kappa$ is a morphism in $\pact{\C}$ from $\alpha$ to $\gamma$, since $\gamma_0 \circ \kappa \circ \iota_0 = id_X = \kappa \circ \alpha_0$ and $\gamma_1 \circ \kappa \circ \iota_1 = 0 = \kappa \circ \alpha_1$.
    
    Suppose that there is a morphism $\kappa' : \beta \to \gamma$ such that 
    \begin{equation*}
        \kappa' \circ \iota = \kappa.
    \end{equation*}
    Then one can check that $\kappa'$ must be given by
    $$\kappa'(x,y) = \kappa(x) + \gamma_1(\kappa(y)) = x + y,$$
    which is not an algebra morphism, since $\kappa'(1,0) \kappa'(0,1) = 1 \neq 0 = \kappa'((1,0)(0,1))$. 
\end{exemplo}

\begin{rem}
    It follows from \cref{exemplo-glob-sem-reflexao,equivalencia-glob-universal-glob} that $\alpha \in \pact{\C}$ does not admit a reflection in $\act{\C}$. 
\end{rem}

\red{
\begin{rem}
   Let $\C$ be the category from \cref{exemplo-glob-sem-reflexao}. Then the partial actions of a group $G$ on objects of $\C$ correspond to the partial actions of $G$ on algebras in the sense of \cite{DE}. Given such a partial action $\alpha$ whose domains are unital ideals, it can be shown that an \textit{enveloping action}~\cite{DE} of $\alpha$ is a universal globalization of $\alpha$.\footnote{The proof of this fact does not fit in this article due to its length.} 
\end{rem}
}

\subsection{Reflection in terms of a colimit}

Now we are going to provide conditions for a partial action datum to have a reflection in $\act{\C}$. To this end, for the remainder of this subsection fix a partial action datum $\alpha(m) = [\dom\alpha_m,\iota_m,\alpha_m]$ of $M$ on $X \in \C$.

Define the category $I$ with $Ob(I) = (M \times M) \sqcup M$, where for each $(m,n) \in Ob(I)$ there is a morphism from $(m,n)$ to $mn$ and a morphism from $(m,n)$ to $m$, and there are no other non-trivial morphisms.

\begin{definicao}\label{funtor-associado}
    The \textit{functor\footnote{Strictly speaking, the functor is not unique, since it depends on the choice of representatives of the isomorphism classes $\af(m)$, $m\in M$.} associated to} $\alpha$ is $F : I \to \C$ that maps $(m,n) \in M \times M$ to $\dom\alpha_n$ and $m \in M$ to $X$, and, given $m,n \in G$, takes the morphism $(m,n) \to mn$ to $\iota_n$ and the morphism $(m,n) \to m$ to $\alpha_n$, as illustrated.

\begin{center}$
    \begin{tikzcd}
  & {(m,n)} \arrow[ld] \arrow[rd] &   \\
mn &                               & m
\end{tikzcd}
\quad
 \begin{tikzcd}[column sep=3.0em]
    {} \arrow[rightsquigarrow]{r}{F} 
    & {}
    \end{tikzcd}
\quad
\begin{tikzcd}
  & \dom\alpha_n \arrow[ld, "\iota_n"'] \arrow[rd, "\alpha_n"] &   \\
X &                                                                          & X
\end{tikzcd}$
\end{center}
 \end{definicao}

We are going to show that a colimit of $F$ induces a reflection of $\alpha$ in $\act{\C}$. To this end, introduce the following notation.

\begin{definicao}
Let $Y \in \C$. We denote by $\Delta(Y)$ the \textit{constant functor} from $I$ to $\C$ that maps all objects in $I$ to $Y$ and all morphisms in $I$ to $id_Y$.
\end{definicao}

\begin{lema}\label{transformacao-natural-permutacao}
Let $\eta = \{\begin{tikzcd}[column sep=small, cramped] F(i) \arrow[r, "\eta_i"] & Y \end{tikzcd} : i \in I\}$ be a natural transformation from $F$ to $\Delta(Y)$. Then for each $m \in M$ the family $\eta^m = \{\begin{tikzcd}[column sep=small, cramped] F(i) \arrow[r, "\eta^m_i"] & Y \end{tikzcd} : i \in I\}$, where
\begin{align}\label{definicao-eta-g}
    \eta^m_i =\begin{cases}
         \eta_{(ms,t)}, & \text{ if } i = (s,t) \in M \times M,\\
         \eta_{ms}, & \text{ if } i = s \in M,
    \end{cases}
\end{align}
is also a natural transformation from $F$ to $\Delta(Y)$.

\end{lema}

\begin{proof}
Fix $m \in M$. Since the only non-trivial morphisms in $I$ are $(s,t) \to st$ and $(s,t) \to s$ for each $s, t \in M$, to verify that $\eta^m$ is a natural transformation from $F$ to $\Delta(Y)$, it suffices to check that the diagrams 
\begin{equation}\label{diagramas-eta_g-transformacao-natural}
\begin{tikzcd}
\dom\alpha_t \arrow[r, "{\eta^m_{(s,t)}}"] \arrow[d, "\iota_t"'] & Y \arrow[d, "id"] \\
X \arrow[r, "\eta^m_{st}"] & Y
\end{tikzcd} \quad
\text{and} \quad
\begin{tikzcd}
\dom\alpha_n \arrow[r, "{\eta^m_{(s,t)}}"] \arrow[d, "\alpha_t"'] & Y \arrow[d, "id"] \\
X \arrow[r, "\eta^m_s"] & Y
\end{tikzcd}
\end{equation}
commute for each $s,t \in M$.

Let $s, t \in M$. Since $\eta$ is a natural transformation from $F$ to $\Delta(Y)$, $\eta_{m(st)} \circ \iota_t = \eta_{(ms)t} \circ \iota_t = \eta_{(ms,t)}$. So, we have the commutativity of the left diagram of \cref{diagramas-eta_g-transformacao-natural}, in view of \cref{definicao-eta-g}. 

The commutativity of the right diagram of \cref{diagramas-eta_g-transformacao-natural} follows similarly.
\end{proof}

Assume that there exists a colimit $\eta: F \rightarrow \Delta(Y)$ of the functor $F$ associated to $\alpha$ and let $\eta^m$ be the corresponding natural transformation from \cref{transformacao-natural-permutacao}.
Given a fixed $m \in M$, by the universal property of $\eta$ there exists a unique natural transformation $\beta_m$ from $\Delta(Y)$ to $\Delta(Y)$ such that $\eta^m = \beta_m \circ \eta$.

That is, for each $m \in M$ there exists a unique morphism \begin{tikzcd}[column sep=small, cramped] Y \arrow[r, "\beta_m"] & Y \end{tikzcd} such that
\begin{equation}\label{beta_g-circ-eta_(a b)}
    \beta_m \circ \eta_{(s,t)} = \eta^m_{(s,t)} = \eta_{(ms,t)}
\end{equation}
for all $s,t \in M$, and
\begin{equation}\label{beta_g-circ-eta_a}
    \beta_m \circ \eta_s = \eta^m_s = \eta_{ms}
\end{equation}
for all $s \in M$.

Consider then the partial action datum $\beta$ of $M$ on $Y$ given by
\begin{equation}\label{beta-definicao}
    \beta(m) = [Y, id_Y, \beta_m] 
\end{equation}
for each $m \in M$.

\begin{lema}
    The partial action datum $\beta$ defined in \cref{beta-definicao} is a global action of $M$ on $Y$.
\end{lema}

\begin{proof}
     We check that $\beta$ satisfies \cref{AG1,AG2}.
     
    \textit{\cref{AG1}.} By \cref{definicao-eta-g},
        $$id_Y \circ \eta_{(s,t)} = \eta_{(s,t)} = \eta_{(es,t)} = \eta^e_{(s,t)} = \beta_e \circ \eta_{(s,t)}$$
        for all $s,t \in M$, and
        $$id_Y \circ \eta_s = \eta_s = \eta_{es} = \eta^e_s=\beta_e \circ \eta_s$$
        for all $s \in M$. Thanks to the uniqueness of $\beta_e$, it follows that $id_Y = \beta_e$. Thus, $\beta(e) = [Y,id_Y,id_Y]$.

    \textit{\cref{AG2}.} Let $m,n \in M$. By \cref{beta_g-circ-eta_(a b),beta_g-circ-eta_a}, for all $s,t \in M$ we have
    \begin{align*}
        (\beta_n \circ \beta_m) \circ \eta_{(s,t)} &= \beta_n \circ (\beta_m \circ \eta_{(s,t)}) = \beta_n \circ \eta_{(ms,t)} = \eta_{(n(ms),t)}\\
        &= \eta_{((nm)s,t)}=\beta_{nm}\circ\eta_{(s,t)}
    \end{align*}
        and
        $$(\beta_n \circ \beta_m) \circ \eta_s = \beta_n \circ (\beta_m \circ \eta_s) = \beta_n \circ \eta_{ms} = \eta_{n(ms)} = \eta_{(nm)s}=\beta_{nm}\circ\eta_s.$$
        Thus, by the uniqueness of $\beta_{nm}$ we have $\beta_n \circ \beta_m = \beta_{nm}$.
\end{proof}

\begin{definicao}\label{acao-global-associada}
    Let $\eta$ be a colimit of the functor associated to $\alpha$. By the \textit{global action associated to} $\eta$ we mean $\beta \in \act{\C}$ given by \cref{beta-definicao}.
\end{definicao}

\begin{proposicao}
    Let $\eta : F \rightarrow \Delta(Y)$ be a colimit of the functor $F$ associated to $\alpha$, and let $\beta$ be the global action associated to $\eta$. Then $\eta_e : X \rightarrow Y$ is a datum morphism from $\alpha$ to $\beta$.
\end{proposicao}

\begin{proof}

    Given $m \in M$, since $\eta$ is a natural transformation from $F$ to $\Delta(Y)$, we have $\eta_m \circ \iota_m = \eta_{em} \circ \iota_m = \eta_{(e,m)} = \eta_e \circ \alpha_m$. By \cref{beta_g-circ-eta_a} we have $\beta_m \circ \eta_e = \eta_m$. Thus,
    $$(\beta_m \circ \eta_e) \circ \iota_m = \eta_m \circ \iota_m = \eta_e \circ \alpha_m,$$
    and so, by \cref{morfismo-de-dados-para-global}, $\eta_e$ is a datum morphism as desired.
\end{proof}

\begin{lema}\label{xi-eh-transformacao-natural}
    Let $(\gamma,f)$ be a pair formed by a global action $\gamma$ of $M$ on $Z \in \C$ and a datum morphism $f : \alpha \rightarrow \gamma$. Then the family $\xi = \{\begin{tikzcd}[column sep=small, cramped] F(i) \arrow[r, "\xi_i"] & Z \end{tikzcd} : i \in I\}$, where
    \begin{equation}\label{xi-definicao-cases}
        \xi_i = \begin{cases}
        \gamma_{mn} \circ f \circ \iota_n, & \text{ if } i = (m,n) \in M \times M,\\
        \gamma_m \circ f, & \text{ if } i = m \in M,
    \end{cases}
    \end{equation}
    is a natural transformation from $F$ to $\Delta(Z)$.    
    
\end{lema}

\begin{proof}

    We shall verify that for each $m, n \in M$ the diagrams
    \begin{equation}\label{diagramas-xi-transformacao-natural}
        \begin{tikzcd}
\dom\alpha_n \arrow[r, "{\xi_{(m,n)}}"] \arrow[d, "\iota_n"'] & Z \arrow[d, "id"] \\
X \arrow[r, "\xi_{mn}"]                                             & Z               
\end{tikzcd} \quad
    \text{and} \quad
        \begin{tikzcd}
\dom\alpha_n \arrow[r, "{\xi_{(m,n)}}"] \arrow[d, "\alpha_n"'] & Z \arrow[d, "id"] \\
X \arrow[r, "\xi_m"]                                                 & Z                
\end{tikzcd}
    \end{equation}
    commute. 

    The commutativity of the left diagram of \cref{diagramas-xi-transformacao-natural} follows directly by \cref{xi-definicao-cases}.
    

    For the second diagram, by \cref{morfismo-de-dados-para-global} we have
    $$f \circ \alpha_n = \gamma_n \circ f \circ \iota_n.$$
 
    Thus, since $\gamma$ is a global action, we have
    \begin{splii}
        \xi_m \circ \alpha_n &= (\gamma_m \circ f) \circ \alpha_n = \gamma_m \circ (f \circ \alpha_n) = \gamma_m \circ (\gamma_n \circ f \circ \iota_n)\\
        &= (\gamma_m \circ \gamma_n) \circ f \circ \iota_n = \gamma_{mn} \circ f \circ \iota_n = \xi_{(m,n)},
    \end{splii}
    giving us the commutativity of the right diagram of \cref{diagramas-xi-transformacao-natural}. 
\end{proof}



\begin{teorema}\label{colimite-implica-reflexao}
    Let \begin{tikzcd}[column sep=small, cramped] F \arrow[r, "\eta"] & \Delta(Y) \end{tikzcd} be a colimit of the functor $F$ associated to $\alpha$ and $\beta$ the global action associated to $\eta$. Then \begin{tikzcd}[column sep=small, cramped] \alpha \arrow[r, "\eta_e"] & \beta \end{tikzcd} is a reflection of $\alpha$ in $\act{\C}$.
\end{teorema}

\begin{proof}
    Let $(\gamma,f)$ be a pair formed by a global action $\gamma$ of $G$ on $Z \in \C$ and a datum morphism $f : \alpha \rightarrow \gamma$. We must show that there exists a unique datum morphism $f' : \beta \rightarrow \gamma$ such that the diagram
    \begin{equation}\label{diagrama-colimite-implica-reflexao}
            \begin{tikzcd}
\alpha \arrow[r, "\eta_e"] \arrow[rd, "f"'] & \beta \arrow[d, "f'", dashed] \\
                                          & \gamma                        
\end{tikzcd}
        \end{equation}
    commutes.

    Let $\xi$ be the natural transformation \cref{xi-definicao-cases} from $F$ to $\Delta(Z)$ constructed in \cref{xi-eh-transformacao-natural}.
    
    By the universal property of $\eta$, there exists a unique morphism \begin{tikzcd}[column sep=small, cramped] Y \arrow[r, "f'"] & Z \end{tikzcd} such that 
    \begin{equation}\label{colimite-implica-reflexao-f'-def}
        \xi_i = f' \circ \eta_i
    \end{equation}
    for each $i \in I$. Since $\gamma_e = id_Z$, by \cref{xi-definicao-cases} we have $\xi_e = f$, whence $f' \circ \eta_e = f$ by \cref{colimite-implica-reflexao-f'-def}. So, diagram \cref{diagrama-colimite-implica-reflexao} commutes modulo the verification that $f'$ is a datum morphism from $\beta$ to $\gamma$, which we are going to do now. By \cref{morfismo-de-dados-para-global}, this will be accomplished if we show that for each $m \in M$
    \begin{equation}\label{f'-circ-betag=gammag-circ-f'}
        f' \circ \beta_m = \gamma_m \circ f'.
    \end{equation}

    To this end, fix $m \in M$ and consider the natural transformation $\xi^m = \{\begin{tikzcd}[column sep=small, cramped] F(i) \arrow[r, "\xi^m_i"] & Z \end{tikzcd} : i \in I\}$ from $F$ to $\Delta(Z)$, where
        \begin{equation}\label{xi-g-i-definicao}
       \xi_i^m = \begin{cases}
            \xi_{(ms,t)}, & \text{ if } i = (s,t) \in M \times M,\\
            \xi_{ms}, & \text{ if } i = s \in M,
        \end{cases}
    \end{equation}
        constructed from $\xi$ as in \cref{transformacao-natural-permutacao}.
    
    By the universal property of $\eta$, there exists a unique morphism \begin{tikzcd}[column sep=small, cramped] Y \arrow[r, "\overline{\xi^m}"] & Z \end{tikzcd} such that for each $i \in I$
    \begin{equation}\label{overline-xi-g-def}
        \xi^m_i = \overline{\xi^m} \circ \eta_i.
    \end{equation}

    Since $\gamma$ is a global action, for all $s \in M$ we have
    \begin{equation}\label{gammag-circ-gammaa=gamma-ga}
        \gamma_m \circ \gamma_s = \gamma_{ms}.
    \end{equation}
    
    Thus, for all $s,t \in M$,
    $$\gamma_m \circ f' \circ \eta_s \overset{\cref{colimite-implica-reflexao-f'-def}}{=} \gamma_m \circ \xi_s \overset{\cref{xi-definicao-cases}}{=} \gamma_m \circ \gamma_s \circ f \overset{\cref{gammag-circ-gammaa=gamma-ga}}{=} \gamma_{ms} \circ f \overset{\cref{xi-definicao-cases}}{=} \xi_{ms} \overset{\cref{xi-g-i-definicao}}{=} \xi^m_s \overset{\cref{overline-xi-g-def}}{=} \overline{\xi^m} \circ \eta_s$$
    and
    \begin{splii}
        \gamma_m \circ f' \circ \eta_{(s,t)} &\overset{\cref{colimite-implica-reflexao-f'-def}}{=} \gamma_m \circ \xi_{(s,t)} \overset{\cref{xi-definicao-cases}}{=} \gamma_m \circ \gamma_{st} \circ f \circ \iota_t \overset{\cref{gammag-circ-gammaa=gamma-ga}}{=} \gamma_{m(st)} \circ f \circ \iota_t \\
        &= \gamma_{(ms)t} \circ f \circ \iota_t \overset{\cref{xi-definicao-cases}}{=} \xi_{(ms,t)} \overset{\cref{xi-g-i-definicao}}{=} \xi^m_{(s,t)} \overset{\cref{overline-xi-g-def}}{=} \overline{\xi^m} \circ \eta_{(s,t)},
    \end{splii}
    so the uniqueness of $\overline{\xi^m}$ in \cref{overline-xi-g-def} gives
    \begin{equation}\label{gamma-g-circ-f'}
        \gamma_m \circ f' = \overline{\xi^m}.
    \end{equation}

    On the other hand, for all $s, t \in M$ we have
    $$f' \circ \beta_m \circ \eta_s \overset{\cref{beta_g-circ-eta_a}}{=} f' \circ \eta_{ms} \overset{\cref{colimite-implica-reflexao-f'-def}}{=} \xi_{ms} \overset{\cref{xi-g-i-definicao}}{=} \xi^m_s \overset{\cref{overline-xi-g-def}}{=} \overline{\xi^m} \circ \eta_s$$
    and
    $$f' \circ \beta_m \circ \eta_{(s,t)} \overset{\cref{beta_g-circ-eta_(a b)}}{=} f' \circ \eta_{(ms,t)} \overset{\cref{colimite-implica-reflexao-f'-def}}{=} \xi_{(ms,t)} \overset{\cref{xi-g-i-definicao}}{=} \xi^m_{(s,t)} \overset{\cref{overline-xi-g-def}}{=} \overline{\xi^m} \circ \eta_{(s,t)},$$
    so the uniqueness of $\overline{\xi^m}$ in \cref{overline-xi-g-def} also gives
    \begin{equation}\label{f'-circ-beta-g}
        f' \circ \beta_m = \overline{\xi^m}.
    \end{equation}
    
    This way, \cref{gamma-g-circ-f',f'-circ-beta-g} complete the proof of \cref{f'-circ-betag=gammag-circ-f'}, so $f'$ is a datum morphism from $\beta$ to $\gamma$, as desired.

    Finally, let us check the uniqueness of $f'$ as a datum morphism from $\beta$ to $\gamma$. To do so, let $f''$ be a datum morphism from $\beta$ to $\gamma$ such that
    \begin{equation}\label{colimite-implica-reflexao-f''-def}
        f = f'' \circ \eta_e.
    \end{equation}

    Since $f''$ is a datum morphism, by \cref{morfismo-de-dados-para-global} we have 
    \begin{equation}\label{f''-circ-betag=gammag-circ-f''}
        f'' \circ \beta_m = \gamma_m \circ f''.
    \end{equation}
    Thus, given $s,t \in G$, we have
    \begin{equation}\label{f''-circ-eta_a=xi_a}
        f'' \circ \eta_s \overset{\cref{beta_g-circ-eta_a}}{=} f'' \circ \beta_s \circ \eta_e \overset{\cref{f''-circ-betag=gammag-circ-f''}}{=} \gamma_s \circ f'' \circ \eta_e \overset{\cref{colimite-implica-reflexao-f''-def}}{=} \gamma_s \circ f \overset{\cref{xi-definicao-cases}}{=} \xi_s.
    \end{equation}

    Since $\eta$ is a natural transformation from $F$ to $\Delta(Y)$, for each $s,t \in M$ we have
    \begin{equation}\label{eta-ab=eta-a-circ-iotab-1a}
        \eta_{(s,t)} = \eta_{st} \circ \iota_t.
    \end{equation}
    
    Hence,
    \begin{equation}\label{f''-circ-eta_ab=ex_ab}
        f'' \circ \eta_{(s,t)} \overset{\cref{eta-ab=eta-a-circ-iotab-1a}}{=} f'' \circ \eta_{st} \circ \iota_t \overset{\cref{f''-circ-eta_a=xi_a}}{=} \xi_{st} \circ \iota_t \overset{\cref{xi-definicao-cases}}{=} \gamma_{st} \circ f \circ \iota_t \overset{\cref{xi-definicao-cases}}{=} \xi_{(s,t)}.
    \end{equation}
    
    So, by \cref{f''-circ-eta_a=xi_a,f''-circ-eta_ab=ex_ab}, $\xi_i = f'' \circ \eta_i$ for all $i \in I$. Since $f'$ is the unique morphism satisfying \cref{colimite-implica-reflexao-f'-def}, we have $f'' = f'$, as desired.
\end{proof}

\begin{corolario}
    Let $\C$ be a cocomplete category. Then $\alpha$ has a universal globalization if and only if $\alpha$ has a (not necessarily universal) globalization.
\end{corolario}

\begin{proof}
    Since $\C$ is cocomplete, the functor $F$ associated to $\alpha$ has a colimit. In this case, by \cref{colimite-implica-reflexao}, $\alpha$ has a reflection in $\act{\C}$. Thus, the result follows by \cref{equivalencia-glob-universal-glob}.
\end{proof}

\subsection{Reflection in terms of coproducts and a coequalizer}

One particular case of a colimit of the functor $F$ gives us a stronger but more tangible condition for a partial action to have a reflection in $\act{\C}$, where we assume that certain coproducts and a certain coequalizer exist in $\C$.

We first introduce the following notation.

\begin{definicao}
    Let $\coprod_{i \in J} X_i$ be a coproduct in $\C$ with inclusions $u_i : X_i \rightarrow \coprod_{i \in J} X_i$ and $\{f_i : X_i \rightarrow Y\}_{i \in J}$ a family of morphisms. Define the \textit{coproduct} of $\{f_i\}_{i \in J}$ (with respect to $u_i$) to be the unique morphism $\coprod_{i \in J} f_i$ from $\coprod_{i \in J} X_i$ to $Y$ such that $(\coprod_{i \in J} f_i) \circ u_i = f_i$ for all $i \in J$.
\end{definicao}

Fix a partial action datum $\alpha(m) = [\dom\alpha_m,\iota_m,\alpha_m]$ of $M$ on $X \in \C$ and assume that the coproducts $\coprod_{m \in M} X$ and $\coprod_{(m,n) \in M \times M} \dom\alpha_n$ exist in $\C$. For each $m, n \in M$, denote the associated inclusion morphisms by
\begin{align}\label{definicao-morfismos-de-inclusao}
    u_m : X \rightarrow \coprod_{m \in M} X \text{ and } u_{(m,n)} : \dom\alpha_n \rightarrow \coprod_{(m,n) \in M \times M} \dom\alpha_n,
\end{align}
and all the coproducts of morphisms in this subsection will be with respect to one of the two families in \cref{definicao-morfismos-de-inclusao}.

Consider the morphisms $p, q : \coprod_{(m,n) \in M \times M} \dom\alpha_n \to \coprod_{m \in M} X$ given as follows:
\begin{align}\label{definicao-p-e-q}
    p = \coprod_{(m,n) \in M \times M} (u_{mn} \circ \iota_n) \text{ and } q = \coprod_{(m,n) \in M \times M} (u_{m} \circ \alpha_n).
\end{align}

We shall now work towards verifying that a coequalizer of $p$ and $q$ induces a colimit of the functor $F$ associated to $\alpha$.

\begin{lema}\label{transformacao-natural-coequaliza-p-q}
    Let $Z \in \C$ and $\xi = \{\begin{tikzcd}[column sep=small, cramped] F(i) \arrow[r, "\xi_i"] & Z \end{tikzcd} : i \in I\}$ be a natural transformation $F \to \Delta(Z)$. Then the coproduct $\coprod_{m \in M} \xi_m : \coprod_{m \in M} X \to Z$ satisfies
    \begin{equation}\label{transformacao-natural-coequaliza-p-q-equacao}
        \left(\coprod_{m \in M} \xi_m\right) \circ p = \left(\coprod_{m \in M} \xi_m\right) \circ q.
    \end{equation}
\end{lema}

\begin{proof}  
    To prove \cref{transformacao-natural-coequaliza-p-q-equacao}, it suffices to verify that for all $m, n \in M$ we have
    $$\left(\coprod_{m \in M} \xi_m\right) \circ p \circ u_{(m,n)} = \left(\coprod_{m \in M} \xi_m\right) \circ q \circ u_{(m,n)}.$$

    

    Denote, for simplicity, $\coprod_{m \in M} \xi_m$ by $\Xi$. Fix $m,n \in M$. Then we have
    $$\Xi \circ (p \circ u_{(m,n)}) = \Xi \circ (u_{mn} \circ \iota_n) = (\Xi \circ u_{mn}) \circ \iota_n = \xi_{mn} \circ \iota_n = \xi_{(m,n)},$$
    where the last equality follows from the fact that $\xi$ is a natural transformation (see \cref{diagramas-xi-transformacao-natural}). Similarly, we have
    $$\Xi \circ (q \circ u_{(m,n)}) = \Xi \circ (u_m \circ \alpha_n) = (\Xi \circ u_m) \circ \alpha_n = \xi_m \circ \alpha_n = \xi_{(m,n)}.$$
\end{proof}

\begin{proposicao}\label{coprod-coequalizador-implica-colimite}
    Assume that there exists a coequalizer \begin{tikzcd}[column sep=small, cramped] \coprod_{m \in M} X \arrow[r, "c"] & Y \end{tikzcd} of $p$ and $q$. Then the family $\eta = \{\begin{tikzcd}[column sep=small, cramped] F(i) \arrow[r, "\eta_i"] & Y \end{tikzcd} : i \in I\}$ such that
        \begin{equation}\label{eta-definicao-coequalizador}
   \eta_i = \begin{cases}
        c \circ p \circ u_{(s,t)} = c \circ q \circ u_{(s,t)}, & \text{ if } i = (s,t) \in M \times M,\\
        c \circ u_s, & \text{ if } i = s \in M,
    \end{cases}
\end{equation}
    is a colimit of the functor $F$ associated to $\alpha$.
\end{proposicao}

\begin{proof}
    Firstly, note that $\eta$ is a natural transformation from $F$ to $\Delta(Y)$, since for each $m, n \in G$ the diagrams
    \begin{equation*}
        \begin{tikzcd}
\dom\alpha_n \arrow[r, "{\eta_{(m,n)}}"] \arrow[d, "\iota_n"'] & Y \arrow[d, "id"] \\
X \arrow[r, "\eta_{mn}"]                                             & Y                
\end{tikzcd} \quad
    \text{and} \quad
        \begin{tikzcd}
\dom\alpha_n \arrow[r, "{\eta_{(m,n)}}"] \arrow[d, "\alpha_n"'] & Y \arrow[d, "id"] \\
X \arrow[r, "\eta_m"]                                                 & Y                
\end{tikzcd}
    \end{equation*}
    commute by \cref{definicao-p-e-q,eta-definicao-coequalizador}.

    Given $Z \in \C$, let $\xi = \{\begin{tikzcd}[column sep=small, cramped] F(i) \arrow[r, "\xi_i"] & Z \end{tikzcd} : i \in I\}$ be a natural transformation from $F$ to $\Delta(Z)$. Let us show that there exists a unique morphism $\varphi : Y \rightarrow Z$ such that 
    \begin{equation}\label{xi-=-varphi-circ-eta}
        \xi = \varphi \circ \eta.
    \end{equation}
    
    By \cref{transformacao-natural-coequaliza-p-q}, we have $\left(\coprod_{m \in M} \xi_m\right) \circ p = \left(\coprod_{m \in M} \xi_m\right) \circ q$, so, by the universal property of $c$ as a coequalizer of $p$ and $q$, there exists a unique morphism $\varphi : Y \to Z$ such that
    \begin{equation}\label{coprodxig=varphi-circ-c}
        \coprod_{m \in M} \xi_m = \varphi \circ c.
    \end{equation}

    Now, $\varphi$ satisfies \cref{xi-=-varphi-circ-eta}, since for all $m, n \in M$ by \cref{coprodxig=varphi-circ-c,eta-definicao-coequalizador} we have
    $$\varphi \circ \eta_m = \varphi \circ (c \circ u_m) = (\varphi \circ c) \circ u_m = \left(\coprod_{m \in M} \xi_m\right) \circ u_m = \xi_m$$
    and by \cref{coprodxig=varphi-circ-c,eta-definicao-coequalizador,definicao-p-e-q} together with the fact that $\xi$ is a natural transformation
    \begin{align*}
        \varphi \circ \eta_{(m,n)} &= \varphi \circ (c \circ p \circ u_{(m,n)}) = (\varphi \circ c) \circ (p \circ u_{(m,n)}) \\
        &= \left(\coprod_{m \in M} \xi_m\right) \circ (u_{mn} \circ \iota_n) = \xi_{mn} \circ \iota_n = \xi_{(m,n)}.
    \end{align*}

    It remains to show that $\varphi$ is the unique morphism satisfying \cref{xi-=-varphi-circ-eta}. For, assume that $\varphi'$ is a morphism such that $\xi = \varphi' \circ \eta$. Then for each $m \in M$ using \cref{eta-definicao-coequalizador} we have
    $$\varphi' \circ c \circ u_m = \varphi' \circ \eta_m = \xi_m,$$
    so that $\varphi' \circ c = \coprod_{m \in M} \xi_m$. Since $\varphi$ is the unique morphism that satisfies \cref{coprodxig=varphi-circ-c}, we conclude that $\varphi' = \varphi$.
\end{proof}

Let $\beta(m) = [Y,id_Y,\beta_m]$ be the global action associated (see \cref{acao-global-associada}) to the colimit $\eta$ from \cref{coprod-coequalizador-implica-colimite}. 
Notice that for each $m \in M$, $\beta_m$ can be described precisely as the unique morphism such that for all $s,t \in M$
\begin{equation}\label{global-action-associated-to-coequalizer}
    \beta_m \circ c \circ u_s = c \circ u_{ms},
\end{equation}
since, in this case,
$$\beta_m \circ c \circ p \circ u_{s,t} = c \circ p \circ u_{ms,t}$$
automatically follows from \cref{global-action-associated-to-coequalizer}.

As a consequence of \cref{coprod-coequalizador-implica-colimite,colimite-implica-reflexao}, we get the following.

\begin{corolario}\label{coeq-implica-reflexao}
    Let \begin{tikzcd}[column sep=small, cramped] \coprod_{m \in M} X \arrow[r, "c"] & Y \end{tikzcd} be a coequalizer of $p$ and $q$. Then $c \circ u_e : \alpha \rightarrow \beta$ is a reflection of $\alpha$ in $\act{\C}$.
\end{corolario}

So, in this case, we can work with a universal globalization of $\alpha$ in terms of coproducts and a coequalizer, due to \cref{equivalencia-glob-universal-glob}.

In a final approach to finding conditions for $\alpha$ to have a reflection in $\act{\C}$, we shall define structures of global actions on the coproducts we worked with so far, in order to find necessary and sufficient conditions in terms of a coequalizer in $\act{\C}$.


For each $m \in M$, consider the morphisms
\begin{equation}\label{definicao-varphi}
    \varphi_m = \coprod_{(s,t) \in M \times M} u_{(ms,t)} : \coprod_{(s,t) \in M \times M} \dom\alpha_t \to \coprod_{(s,t) \in M \times M} \dom\alpha_t
\end{equation}
     and
     \begin{equation}\label{definicao-psi}
         \psi_m = \coprod_{s \in M} u_{ms} : \coprod_{s \in M} X \to \coprod_{s \in M} X.
     \end{equation}


It is a simple verification that the partial action data
$$\varphi(m) = [\textstyle\coprod_{(s,t) \in M \times M} \dom\alpha_t,id,\varphi_m] \text{ and } \psi(m) = [\textstyle\coprod_{s \in M} X,id,\psi_m]$$
are global actions of $M$ on $\coprod_{(m,n) \in M \times M} \dom\alpha_n$ and $\coprod_{m \in M} X$, respectively.

\begin{proposicao}
    The morphisms $p,q :\coprod_{(m,n) \in M \times M} \dom\alpha_n \to \coprod_{m \in M} X$ given by \cref{definicao-p-e-q} are morphisms from $\varphi$ to $\psi$ in $\act{\C}$.
\end{proposicao}

     \begin{proof}
     We shall only verify that $q$ is a datum morphism, as the verification for $p$ is analogous. By \cref{morfismo-de-dados-para-global}, it suffices to show that
     $$q \circ \varphi_m = \psi_m \circ q,$$
     for each $m \in M$. The latter is equivalent to
     $$q \circ \varphi_m \circ u_{(s,t)} = \psi_m \circ q \circ u_{(s,t)} \text{ for each } (s,t) \in M \times M.$$
     Indeed, fixed $m \in M$, by \cref{definicao-p-e-q,definicao-varphi,definicao-psi}, for all $s, t\in M$ we have
     \begin{splii}
         q \circ \varphi_m \circ u_{(s,t)} &= q \circ u_{(ms,t)} = u_{ms} \circ \alpha_t = (\psi_m \circ u_s) \circ \alpha_t \\
         &= \psi_m \circ (u_s \circ \alpha_t)
         = \psi_m \circ q \circ u_{(s,t)}.
     \end{splii}
     \end{proof}

\begin{lema}\label{H-esta-bem-definida}
Let $\gamma$ be a global action of $M$ on $Z \in \C$ and \begin{tikzcd}[column sep=small, cramped] X \arrow[r, "f"] & Z \end{tikzcd} a morphism in $\C$. Then $\coprod_{m \in M} (\gamma_m \circ f)$ is a morphism from $\psi$ to $\gamma$ in $\act{\C}$.
\end{lema}

\begin{proof}
    For the simplicity of notation, let $\Gamma = \coprod_{m \in M} (\gamma_m \circ f)$. Then for each $s \in M$ we have
    $$\Gamma \circ u_s = \gamma_s \circ f.$$
    Therefore, given $m,s \in M$, by \cref{definicao-psi} and the fact that $\gamma$ is a global action we have
\begin{splii}
    (\Gamma \circ \psi_m) \circ u_s &= \Gamma \circ (\psi_m \circ u_s) = \Gamma \circ u_{ms} = \gamma_{ms} \circ f = (\gamma_m \circ \gamma_s) \circ f \\
    &= \gamma_m \circ (\gamma_s \circ f) = \gamma_m \circ (\Gamma \circ u_s)
    = (\gamma_m \circ \Gamma) \circ u_s,
\end{splii}
so that $\Gamma \circ \psi_m = \gamma_m \circ \Gamma$.

Thus, by \cref{morfismo-de-dados-para-global}, $\Gamma$ is a datum morphism from $\psi$ to $\gamma$, as desired.
\end{proof}

In view of \cref{H-esta-bem-definida} we can define the following map.

\begin{definicao}\label{function-definition-H}

Let $\gamma$ be a global action of $M$ on $Z \in \C$. Define $H_{X,\gamma}: \Hom{\C}(X,Z) \rightarrow \Hom{\act{\C}}(\psi,\gamma)$ by
\begin{align}\label{function-definition-H-explicita}
    H_{X,\gamma}(f) = \coprod_{m \in M} (\gamma_m \circ f)
\end{align}
 for any $f \in \Hom{\C}(X,Z)$.
\end{definicao}

\begin{proposicao}\label{bijection-morphisms-morphisms-of-coproduct-data}
Let $\gamma$ be a global action of $M$ on $Z \in \C$. Then $H_{X,\gamma}$ given by \cref{function-definition-H-explicita} is a bijection whose inverse is
\begin{align}\label{inversa-H-definicao}
    \Hom{\act{\C}}(\psi,\gamma)\ni\Gamma \mapsto \Gamma \circ u_e \in \Hom{\C}(X,Z).
\end{align}
\end{proposicao}

\begin{proof}
   To simplify the notation, denote $H_{X,\gamma}$ by $H$. Let us verify that the map $G$ defined in \cref{inversa-H-definicao} is the inverse of $H$.

Let $f \in \Hom{\C}(X,Z)$. Then by \cref{inversa-H-definicao,function-definition-H-explicita}
$$G(H(f))=\coprod_{m\in M}(\gamma_m\circ f)\circ u_e=\gamma_e\circ f=f.$$

On the other hand, let $\Gamma \in \Hom{\act{\C}}(\psi,\gamma)$. We have by \cref{inversa-H-definicao,function-definition-H-explicita}
\begin{align}\label{H(M(gamma))-explicitado}
    H(G(\Gamma))=\coprod_{m\in M}(\gamma_m\circ(\Gamma\circ u_e)).
\end{align}

As $\Gamma$ is a datum morphism, for each $m\in M$ we have $\Gamma \circ \psi_m = \gamma_m \circ \Gamma$, so, by \cref{definicao-psi,H(M(gamma))-explicitado},
    $$\Gamma \circ u_m = \Gamma \circ \psi_m \circ u_e = \gamma_m \circ \Gamma \circ u_e = H(G(\Gamma)) \circ u_m.$$
    Thus, $\Gamma = H(G(\Gamma))$.
    %
\end{proof}

\red{\begin{rem}
    Let $U$ be the forgetful functor from $\act{\C}$ to $\C$ and $L$ the functor that maps $X \in \C$ to the global action of $M$ on $\coprod_{m \in M} X$ by shift of indices and $f \in \Hom{\C}(X,Y)$ to $\coprod_{m \in M} (v_m \circ f) \in \Hom{\act{\C}}(L(X),L(Y))$ (where the morphisms $v_m$ are the canonical inclusions of $Y$ into $\coprod_{m \in M} Y$). Then the isomorphisms $H_{X,\gamma}$ form a natural isomorphism
    $$H : \Hom{\C}(\_,U(\_)) \to \Hom{\act{\C}}(L(\_),\_),$$
    so $U$ and $L$ are adjoint functors.
\end{rem}}


\begin{lema}\label{H-da-composicao}
    Let $\beta$ and $\gamma$ be global actions of $M$ on $Y \in \C$ and $Z \in \C$, respectively, $g \in \Hom{\act{\C}}(\beta,\gamma)$ and $f\in\Hom{\C}(X,Y)$. Then
    \begin{equation}\label{H-g-circ-f=g-circ-H-f}
        H_{X,\gamma}(g \circ f) = g \circ H_{X,\beta}(f).
    \end{equation}
\end{lema}

\begin{proof}
    Let $m \in M$. By \cref{morfismo-de-dados-para-global} we have $\gamma_m \circ g = g \circ \beta_m$, so by \cref{function-definition-H-explicita}
    $$H_{X,\gamma}(g \circ f) \circ u_m = \gamma_m \circ g \circ f = g \circ \beta_m \circ f = g \circ H_{X,\beta}(f) \circ u_m,$$
    whence \cref{H-g-circ-f=g-circ-H-f}.
\end{proof}
    

\begin{lema}\label{bijecao-morfismos-de-dados-morfismos-de-dados-do-coproduto-que-coequalizam}
Let $\gamma$ be a global action of $M$ on $Z \in \C$ and $f \in \Hom\C(X,Z)$. Then $f$ is a datum morphism from $\alpha$ to $\gamma$ if and only if
\begin{align}\label{H-A-gamma(f)-circ-p=H-A-gamma(f)-circ-q}
    H_{X, \gamma}(f) \circ p = H_{X, \gamma}(f) \circ q.
\end{align}

\end{lema}

\begin{proof}
    Write $H=H_{X,\gamma}$ for short. Assume that $f$ is a datum morphism from $\alpha$ to $\gamma$. For all $m \in M$, by \cref{morfismo-de-dados-para-global} we have
    \begin{equation}\label{gammag-f-iotag=f-alphag}
        \gamma_m \circ f \circ \iota_m = f \circ \alpha_m,
    \end{equation}
and by \cref{function-definition-H-explicita} we have
\begin{equation}\label{H(f)-ug=gammag-f}
    H(f) \circ u_m = \gamma_m \circ f.
\end{equation}

By \cref{H(f)-ug=gammag-f,definicao-p-e-q} we get
\begin{equation}\label{H(f)-q-uab=gammab-f-alphab-1a}
    H(f) \circ q \circ u_{(s,t)} = H(f) \circ u_s \circ \alpha_t = \gamma_s \circ f \circ \alpha_t.
\end{equation}

Therefore, by \cref{gammag-f-iotag=f-alphag,H(f)-q-uab=gammab-f-alphab-1a,H(f)-ug=gammag-f,definicao-p-e-q}, and since $\gamma$ is a global action, for each $(s,t) \in M \times M$,
\begin{align*}
    H(f) \circ q \circ u_{(s,t)} &\overset{\cref{H(f)-q-uab=gammab-f-alphab-1a}}{=} \gamma_s \circ f \circ \alpha_t = \gamma_s \circ (f \circ \alpha_t) \overset{\cref{gammag-f-iotag=f-alphag}}{=} \gamma_s \circ (\gamma_t \circ f \circ \iota_t) \\
    &= (\gamma_s \circ \gamma_t) \circ f \circ \iota_t = \gamma_{st} \circ f \circ \iota_t = (\gamma_{st} \circ f) \circ \iota_t\\
    &\overset{\cref{H(f)-ug=gammag-f}}{=} (H(f) \circ u_{st}) \circ \iota_t = H(f) \circ (u_{st} \circ \iota_t) \overset{\cref{definicao-p-e-q}}{=} H(f) \circ p \circ u_{(s,t)},
\end{align*}
whence \cref{H-A-gamma(f)-circ-p=H-A-gamma(f)-circ-q}.

Conversely, assume that $\Gamma \coloneqq H(f)$ satisfies
\begin{equation}\label{Gamma-p=Gamma-q}
    \Gamma \circ p = \Gamma \circ q.
\end{equation}
By \cref{bijection-morphisms-morphisms-of-coproduct-data},
\begin{equation}\label{f=Gamma-ue}
    f = H(f) \circ u_e = \Gamma \circ u_e.
\end{equation}

It follows from \cref{f=Gamma-ue,Gamma-p=Gamma-q,definicao-p-e-q,H(f)-ug=gammag-f} that for all $m \in M$
\begin{splii}
    (\gamma_m \circ f) \circ \iota_m &\overset{\cref{H(f)-ug=gammag-f}}{=} (\Gamma \circ u_m) \circ \iota_m = \Gamma \circ (u_m \circ \iota_m) \overset{\cref{definicao-p-e-q}}{=} \Gamma \circ (p \circ u_{(e,m)}) \\
    &= (\Gamma \circ p) \circ u_{(e,m)} \overset{\cref{Gamma-p=Gamma-q}}{=} (\Gamma \circ q) \circ u_{(e,m)} = \Gamma \circ (q \circ u_{(e,m)}) \\
    &\overset{\cref{definicao-p-e-q}}{=} \Gamma \circ (u_e \circ \alpha_m) = (\Gamma \circ u_e) \circ \alpha_m \overset{\cref{f=Gamma-ue}}{=} f \circ \alpha_m,
\end{splii}
and, thus, by \cref{morfismo-de-dados-para-global}, $f$ is a datum morphism from $\alpha$ to $\gamma$.
\end{proof}

\begin{teorema}\label{reflector<->coequalizer}
    The following statements hold:



\begin{enumerate}
\item\label{teorema-reflexao-sse-coeq-item-i} If \begin{tikzcd}[column sep=small, cramped] \alpha \arrow[r, "r"] & \beta \end{tikzcd} is a reflection of $\alpha$ in $\act{\C}$, then $H_{X,\beta}(r)$ is a coequalizer of $p$ and $q$ in $\act{\C}$.

\item\label{teorema-reflexao-sse-coeq-item-ii} If \begin{tikzcd}[column sep=small, cramped] \psi \arrow[r, "c"] & \beta \end{tikzcd} is a coequalizer of $p$ and $q$ in $\act{\C}$, then $c \circ u_e$ is a reflection of $\alpha$ in $\act{\C}$.

\end{enumerate}

In particular, $\alpha$ has a reflection in $\act{\C}$ if and only if $p$ and $q$ have a coequalizer in $\act{\C}$.
\end{teorema}

\begin{proof}
    \cref{teorema-reflexao-sse-coeq-item-i} Assume \begin{tikzcd}[column sep=small, cramped] \alpha \arrow[r, "r"] & \beta \end{tikzcd} is a reflection of $\alpha$ in $\act{\C}$. Let us check that $H_{X,\beta}(r)$ is a coequalizer of $p$ and $q$ in $\act{\C}$.

   By \cref{bijecao-morfismos-de-dados-morfismos-de-dados-do-coproduto-que-coequalizam,H-esta-bem-definida}, $H_{X,\beta}(r)$ is a datum morphism from $\psi$ to $\beta$ such that $H_{X,\beta}(r) \circ p = H_{X,\beta}(r) \circ q$.
    
    Let \begin{tikzcd}[column sep=small, cramped] \psi \arrow[r, "f"] & \gamma \end{tikzcd} be a datum morphism such that $f \circ p = f \circ q$. We must show that there exists a unique datum morphism $f' : \beta \rightarrow \gamma$ such that
    \begin{equation}\label{f-=-f'-circ-H-r}
        f = f' \circ H_{X,\beta}(r).
    \end{equation}
    
    By \cref{bijecao-morfismos-de-dados-morfismos-de-dados-do-coproduto-que-coequalizam}, $H_{X,\gamma}^{-1}(f)=f \circ u_e$ is a datum morphism from $\alpha$ to $\gamma$. Since $r$ is a reflection of $\alpha$ in $\act{\C}$, there exists a unique datum morphism $f' : \beta \rightarrow \gamma$ such that
    \begin{align}\label{f'-circ-r=f-circ-u_e}
        f' \circ r = f \circ u_e.
    \end{align}

    This way, by \cref{bijection-morphisms-morphisms-of-coproduct-data,f'-circ-r=f-circ-u_e,H-da-composicao} we have
    \begin{equation*}
        f = H_{X,\gamma}(f \circ u_e) = H_{X,\gamma}(f' \circ r) = f' \circ H_{X,\beta}(r).
    \end{equation*}

    Note that $f'$ is the unique datum morphism satisfying \cref{f-=-f'-circ-H-r}. Indeed, if $f''$ is a datum morphism from $\beta$ to $\gamma$ such that $f'' \circ H_{X,\beta}(r) = f$, then by \cref{bijection-morphisms-morphisms-of-coproduct-data}
    $$f'' \circ r = f'' \circ (H_{X,\beta}(r) \circ u_e) = (f'' \circ H_{X,\beta}(r)) \circ u_e = f \circ u_e,$$
    and by the uniqueness of $f'$ in \cref{f'-circ-r=f-circ-u_e}, we have $f'' = f'$. Thus, $H_{X,\beta}(r)$ is a coequalizer of $p$ and $q$ in $\act{\C}$.

    \cref{teorema-reflexao-sse-coeq-item-ii} Let $c$ be a coequalizer of $p$ and $q$ in $\act{\C}$. We shall verify that $c \circ u_e$ is a reflection of $\alpha$ in $\act{\C}$.

    By \cref{bijecao-morfismos-de-dados-morfismos-de-dados-do-coproduto-que-coequalizam}, $H^{-1}(c) = c \circ u_e$ is a datum morphism from $\alpha$ to $\beta$, since $c$ is a datum morphism such that $c \circ p = c \circ q$.

    So, let \begin{tikzcd}[column sep=small, cramped] \alpha \arrow[r, "f"] & \gamma \end{tikzcd} be a datum morphism. We must show that there exists a unique datum morphism $f'$ from $\beta$ to $\gamma$ such that
    \begin{equation}\label{f-=-f'-circ-c-circ-ue}
        f = f' \circ (c \circ u_e).
    \end{equation}

    \cref{bijecao-morfismos-de-dados-morfismos-de-dados-do-coproduto-que-coequalizam} tells us that $H_{X,\gamma}(f)$ is a datum morphism from $\psi$ to $\gamma$ such that $H_{X,\gamma}(f) \circ p = H_{X,\gamma}(f) \circ q$. So, since $c$ is a coequalizer of $p$ and $q$, there exists a unique datum morphism $f'$ from $\beta$ to $\gamma$ such that 
    \begin{align}\label{H_A_gamma(f)=f'-circ-c}
        H_{X,\gamma}(f) = f' \circ c.
    \end{align}
    This morphism is such that
    $$f = H_{X,\gamma}(f) \circ u_e = (f' \circ c) \circ u_e = f' \circ (c \circ u_e).$$

    Moreover, $f'$ is the unique datum morphism satisfying \cref{f-=-f'-circ-c-circ-ue}. Indeed, if there were $f''$ with $f = f'' \circ (c \circ u_e)$, then by \cref{H-da-composicao,bijection-morphisms-morphisms-of-coproduct-data}
    $$H_{X,\gamma}(f) = H_{X,\gamma}(f'' \circ (c \circ u_e)) = f'' \circ H_{X,\beta}(c \circ u_e) = f'' \circ c,$$
    and, due to the uniqueness of $f'$ in \cref{H_A_gamma(f)=f'-circ-c}, we would have $f'' = f'$.
\end{proof}

\red{In the following remark we compare \cref{reflector<->coequalizer,equivalencia-glob-universal-glob} with (the dual of) \cite[Theorem 3.5]{saracco2022globalization}.}

\red{\begin{rem}
     In $\Set$ the coproducts $\coprod_{m \in M} X$ and $\coprod_{(m,n) \in M \times M} \dom\alpha_n$ are isomorphic to $X \times M$ and $(X \bullet M) \times M$ (as in the notation of \cite{saracco2022globalization}), respectively. Via these isomorphisms, the morphisms $p$ and $q$ in \cref{definicao-p-e-q} correspond to (the duals of) the lower and upper arrows in \cite[diagram (6)]{saracco2022globalization}, respectively.
\end{rem}}


\section{Application to partial monoid actions on sets}\label{sec-pact-Set}

Fix $\alpha(m) = [\dom\alpha_m,\iota_m,\alpha_m]$ a partial action datum of $M$ on a set $X$, with $\dom\alpha_m \subseteq X$ and $\iota_m$ the respective inclusion. Denote by $\approx$ the equivalence relation on $M \times X$ generated by $\sim$, where
     \begin{align}\label{acao-parcial-rel-equiv-descricao}
     (m,x) \sim (n,y)\iff \exists m' \in M \text{ such that } m = nm', x \in \dom\alpha_{m'} \text{ and } y = \alpha_{m'}(x).
     \end{align}
Let $Y = (M \times X)/{\approx}$ and denote by $[m,x]$ the $\approx$-equivalence class of $(m,x)$.

\begin{lema}\label{reflexao-dado-acao-parcial-set}

     
     The maps $\beta_n : Y \to Y$ given by 
     \begin{equation}\label{beta-n-[m-x]=[nm-x]}
         \beta_n([m,x]) = [nm,x]
     \end{equation}
    define a global action $\beta$ of $M$ on $Y$ and the map $\iota : X \to Y$ given by 
    \begin{equation*}
        \iota(x) = [e,x]
    \end{equation*}
    is a reflection $\alpha \to \beta$ of $\alpha$ in $\act{\Set}$.
\end{lema}

\begin{proof}
    Consider the coproduct $\coprod_{m \in M} X = M \times X$ with inclusions $u_m : X \ni x \mapsto (m,x) \in M \times X$, and the coproduct $\coprod_{(m,n) \in M \times M} \dom\alpha_n = \{(m,n,x) : x \in \dom\alpha_n\} \eqqcolon M^2 \bullet X$ with inclusions $u_{(m,n)} : \dom\alpha_n \ni x \mapsto (m,n,x) \in M^2 \bullet X$. Then the maps $p,q : M^2 \bullet X \to M \times X$ from \cref{definicao-p-e-q} are given by
    $$p(m,n,x) = (mn,\iota_n(x)) = (mn,x)$$
    and
    $$q(m,n,x) = (m,\alpha_n(x)).$$

    The canonical projection $c$ of $M \times X$ onto its quotient by the equivalence relation generated by $\sim = \{(p(m,n,x), q(m,n,x)) : (m,n,x) \in M^2 \bullet X\}$ is a coequalizer of $p$ and $q$. It is a simple verification that $\sim$ coincides with \cref{acao-parcial-rel-equiv-descricao}, so $c$ is precisely the natural projection of $M \times X$ onto $Y$.
    
    Then the global action $\beta$ of $M$ on $Y$ from \cref{coeq-implica-reflexao} (see \cref{global-action-associated-to-coequalizer}) is given precisely by \cref{beta-n-[m-x]=[nm-x]}, with $\iota = c \circ u_e : \alpha \to \beta$ being a reflection of $\alpha$ in $\act{\Set}$.
\end{proof}

\begin{proposicao}\label{af-glob-iff-af-strong-for-C=Set}
    A partial action datum $\alpha$ in $\datum{\Set}$ has a universal globalization if and only if $\alpha \in \spact{\Set}$.
\end{proposicao}

\begin{proof}
    If $\alpha$ has a universal globalization, then $\alpha \in \spact{\Set}$ by \cref{remark-restricao-eh-acao-parcial-forte}.

    Conversely, by \cref{reflexao-dado-acao-parcial-set} $\alpha$ has a reflection in $\act{\C}$, which, since $\alpha \in \spact{\Set}$, by \red{\cite[Lemma 5.5]{PartMonoids}}, is such that \cref{diagrama-equivalencia-glob-universal-glob} is a pullback for all $m \in M$. Thus, by \cref{equivalencia-glob-universal-glob} $\alpha$ has a universal globalization.
\end{proof}

	\section*{Acknowledgements}
	The authors are grateful to Eliezer Batista, Ganna Kudryavtseva, Paolo Saracco and Joost Vercruysse for fruitful discussions and useful comments and to Mikhailo Dokuchaev for the references~\cite{Green-Marcos94,Palais}. \blue{We also thank the referee for the very careful reading and valuable suggestions that resulted in a significant improvement of the original manuscript.}
	 The first author was partially supported by CMUP, member of LASI, which is financed by national funds through FCT --- Fundação para a Ciência e a Tecnologia, I.P., under the project with reference UIDB/00144/2020.
        The second author was supported by FAPESC.
	\bibliography{Ref}{}
	\bibliographystyle{acm}
	
\end{document}